\documentclass[11pt]{amsart}
\usepackage{amsmath,amsthm,amssymb,amsfonts,enumerate,color,hyperref,esint, graphicx, pgf,tikz}
\usepackage[normalem]{ulem}
\usepackage[latin1]{inputenc}
\usepackage{appendix,verbatim}

\newcommand{\R}{\mathbb{R}}
\newcommand{\Rn}{\mathbb{R}^n}

\newcommand{\N}{\mathbb{N}}

\newcommand{\M}{\mathcal{M}}

\renewcommand{\S}{\mathcal{S}}
\newcommand{\vep}{\varepsilon}
\newcommand{\tr}{\operatorname{tr}}

\newtheorem{thm}{Theorem}[section]
\newtheorem{prop}[thm]{Proposition}
\newtheorem{cor}[thm]{Corollary}
\newtheorem{lem}[thm]{Lemma}
\theoremstyle{definition}
\newtheorem{defn}[thm]{Definition}
\newtheorem{rem}[thm]{Remark}

\numberwithin{equation}{section}

\title{A fully nonlinear parabolic segregation problem}

\author[E.~Casey]{Emily Casey}
\address{Emily Casey. School of Mathematics, University of Minnesota \newline
127 Vincent Hall 206 Church St. SE Minneapolis, MN 55455, USA}
    \email{\tt ecasey@umn.edu}
    
\author[C.~Mihaila]{Cornelia Mihaila}
\address{Cornelia Mihaila. Saint Michael's College\newline
One Winooski Park, Colchester, VT, 05439}
    \email{\tt cmihaila@smcvt.edu}
    
\author[M.~Soria-Carro]{Mar\'ia Soria-Carro}
\address{Mar\'ia Soria-Carro. Departamento de Matem\'aticas, Universidad Aut\'onoma de Madrid\newline
Campus de Cantoblanco, 28049 Madrid, Spain}
    \email{\tt maria.soriacarro@uam.es}
    
\author[C.~Torres-Latorre]{Clara Torres-Latorre}
\address{Clara Torres-Latorre. Instituto de Ciencias Matem\'aticas\newline Consejo Superior de Investigaciones Cient\'ificas\newline
    C/ Nicol\'as Cabrera, 13-15, 28049 Madrid, Spain.}
\email{\tt clara.torres@icmat.es}



\begin{document}

\maketitle

\begin{abstract}
We study a two-phase parabolic free boundary problem arising as the segregation limit of a competition-diffusion system with strong interactions and fully nonlinear diffusion. The positive and negative phases are governed by different equations and are coupled through global viscosity inequalities. We construct viscosity solutions that are Lipschitz in space and $1/2$-H\"older in time, and establish differentiability in space across the free boundary, including at singular points. At regular points, we prove that the solution has a parabolic $C^{1,\alpha}$ asymptotic expansion, and deduce that the free boundary is locally a $C^{1,\alpha}$ hypersurface. The proof relies on a parabolic Alt--Caffarelli--Friedman monotonicity formula to classify blow-up limits, followed by an improvement of flatness argument.
\end{abstract}

\section{Introduction} 

Segregation phenomena occur when strong interactions prevent different populations or phases from coexisting in the same region. As the strength of the interaction increases, the different components tend to separate and occupy disjoint regions. The interface between these regions is not known in advance and is determined as the populations segregate. Hence, when modeling such populations with partial differential equations (PDEs), the limiting configuration is naturally described by a free boundary problem. These problems arise in several areas, such as population dynamics, materials science, and quantum systems. A central question in this area of PDE theory is to understand how the equations governing the individual phases determine the geometry and regularity of the interface separating them.

In this paper, we consider the following fully nonlinear parabolic two-phase problem
\begin{equation} \label{eq:FBP}
\begin{cases}
\partial_t u- F^-(D^2 u)  = 0 & \text{ in } \{u>0\},\\
\partial_t u-F^+(D^2 u)= 0 & \text{ in } \{u<0\},\\
\partial_t u- F^-(D^2 u)  \geq 0 & \text{ in } \Omega\times (0,T],\\
 \partial_t u - F^+(D^2 u)   \leq 0 & \text{ in } \Omega \times (0,T],\\
 u(x,0) = f(x) & \text{ for } x\in \Omega,\\
 u(x,t) = g(x,t) & \text{ for } (x,t)\in \partial \Omega \times (0, T],
\end{cases}
\end{equation}
where $\Omega\subset \Rn$ is a smooth bounded domain and $T>0$. 
The functions $f:\overline \Omega \to \R$ and $g: \partial\Omega \times [0,T] \to \R$ are continuous and satisfy the compatibility condition $f=g$ on $\partial \Omega\times\{0\}$.
The operators $F^+$ and $F^-$ are fully nonlinear operators, whose properties will be discussed in Section~\ref{sec:operators}.
Our main object of interest is the \textit{free boundary}, which we define as the set
\begin{equation} \label{eq:Gamma}
\Gamma:=\partial \{u>0\} \cap (\Omega\times (0,T]).
\end{equation}
We point out that when $F^+=F^-=F$, the problem reduces to a single equation
$$\partial_t u - F(D^2 u)=0,$$ 
for which regularity results are already available \cite{WI92,WII92}. In this case, the free boundary $\Gamma$ is the nodal set $\{u=0\}$ of the solution $u$. By contrast, a key feature of \eqref{eq:FBP} is that the positive and negative phases evolve according to \textit{different} fully nonlinear parabolic equations. 
The two phases are coupled through the global viscosity inequalities, which implicitly determine the behavior of $u$ across the free boundary $\Gamma$. In fact, we will see that the spatial gradient $\nabla u(x_0,t_0)$ exists at every point $(x_0,t_0)\in \Gamma$.

Problem~\eqref{eq:FBP} arises as the limiting configuration of a competition-diffusion system of two species with strong interactions and nonlinear diffusion.
The time-independent problem was studied by Caffarelli, Patrizi, Quitalo, and Torres in \cite{CPQT19} (see also \cite{Q13}), in which they construct a Lipschitz viscosity solution and prove that the \textit{regular} set of the free boundary is locally a $C^{1,\alpha}$ hypersurface.
Motivated by their work, our goal is to extend these results to the parabolic framework, where the free boundary $\Gamma$ is now a space-time interface, see \eqref{eq:Gamma}. 
This extension presents new challenges, both in obtaining uniform space-time regularity estimates and in understanding the behavior of the evolving interface.

\subsection{History and motivation}
This particular segregation problem has roots in mathematical ecology and the study of population dynamics of competing species inhabiting the same region. One such model for this phenomenon is the Gause--Lotka--Volterra equations, introduced in the 1920s, for $N$ species, with densities $u_i \geq 0$, $i=1,\dots,N$:
\begin{equation}\label{e: GLV}
    \partial_t u_i
\quad =\quad
\underbrace{d_i \Delta u_i}_{\text{diffusion}}
\quad + \quad \underbrace{f(u_i)}_{\text{reaction}}
\quad - \quad
\underbrace{\sum_{j\neq i} {b_{ij}} u_i u_j}_{\text{competition}},
\end{equation}
where $d_i$ is the rate of diffusion and $b_{ij}$ is the rate of competition between species $i$ and $j$. There has been a lot of activity in the study of 
\eqref{e: GLV}, where the relationship between the parameters plays an important role in the behavior of the species (see also the Fisher equation, e.g., \cite{CKL09}).  
We discuss some of these works in more detail later in this section. 

Most relevant to the study of \eqref{eq:FBP} is the reduction of \eqref{e: GLV} to a competition-diffusion model, namely where $f(u_i)\equiv 0$ for $i=1,\hdots, N$. 
We are interested in the strong competition regime, i.e., when $b_{ij}\to \infty$.
A model of this type appeared in \cite{CKL09}, although the study of competition-diffusion models with strong interactions and their asymptotic limits began to develop in the 1990s.

To illustrate why interfaces appear naturally in this setting,  we consider a toy model of two competing species. For $\varepsilon>0$, let $u_1^\varepsilon$ and $u_2^\varepsilon$ denote their densities. Assume that both populations live in the same region $\Omega\subset \Rn$, and their densities satisfy the system
\begin{equation*}
\begin{cases}
\Delta u_1^\varepsilon= \dfrac{1}{\varepsilon} u_1^\varepsilon u_2^\varepsilon & \text{in }\Omega,\\[4mm]
\Delta u_2^\varepsilon= \dfrac{1}{\varepsilon} u_1^\varepsilon u_2^\varepsilon & \text{in }\Omega,\\[4mm]
u_1^\varepsilon=f_1, \ u_2^\varepsilon=f_2  & \text{on }\partial\Omega,\\
\end{cases}
\end{equation*}
where the boundary data have disjoint support, i.e.,
$$
f_1f_2\equiv 0.
$$
The parameter $\vep$ controls the strength of the competition: when $\varepsilon$ is small, the interaction term becomes large wherever both populations coexist, thus penalizing these regions. Formally, as $\varepsilon\to0$, the densities converge to limiting functions $u_1$ and $u_2$ satisfying
$$
u_1u_2 \equiv 0.
$$
Hence, the two populations occupy disjoint regions separated by an a priori unknown interface.  This limit is known in the literature as the \textit{spatial segregation limit.}

Introducing the signed function
$$
u:=u_1-u_2,
$$
and noting that $u_1$ and $u_2$ are nonnegative and have disjoint supports, one obtains 
$$
u_1=u^+ \quad \text{ and } \quad u_2=u^-,
$$
where $u^+$ and $u^-$ are the positive and negative parts of $u$, respectively.
The set $\{u>0\}$ represents the region occupied by the first population, and the set $\{u<0\}$ is the region occupied by the second. Their common interface becomes the free boundary $\partial\{u>0\}$.

For this two-species model, the limiting function $u=u_1-u_2$ is uniquely determined by the boundary data. Indeed, the competition terms cancel when the two equations are subtracted, thus $\Delta(u_1^\varepsilon-u_2^\varepsilon)=0$. Passing to the limit, we obtain $\Delta u=0$ in $\Omega$. Hence, $u$ is the harmonic function in $\Omega$ with boundary data $f_1-f_2$, and the interface separating the two populations is the nodal set $\{u=0\}\cap\Omega$, whose regularity is nowadays well understood; see the survey~\cite{H07}. 
This reduction to a single harmonic function does not work when three or more populations interact~\cite{CKL09}, or when diffusion is governed by nonlinear operators~\cite{Q13,CPQT19}. In these settings, the limit is described by a system of equations and inequalities, such as~\eqref{eq:FBP}, and the regularity of the interface must be studied via the corresponding free boundary problem.

Extensive research has focused on rigorously establishing the spatial segregation limit and understanding the properties of the resulting segregated configurations in different models. These questions were initially studied for reaction--diffusion models arising in population dynamics by Dancer and Du \cite{DD94,DD95I,DD95II}, who proved existence of positive solutions, including systems with three competing species. Dancer, Hilhorst, Mimura, and Peletier  \cite{DHMP99} later proved that, for two species, strong competition in \eqref{e: GLV} leads to a segregated limit described by a free boundary problem. For subsequent contributions, see \cite{DD03,CDHMN04} and the references therein.

Conti, Terracini, and Verzini \cite{CTV05I,CTV05II,CTV05III} then developed a variational approach for multiple species. In their work, they connected the spatial segregation limit with optimal partition problems and proved existence, Lipschitz regularity, and uniqueness. In two dimensions, they also described the interface as a collection of regular curves that may meet at isolated multiple junctions. We refer to \cite{OV26} for a recent survey on optimal partition problems.

Further regularity results were obtained by Caffarelli and Lin in the variational setting \cite{CL08}, and by Caffarelli, Karakhanyan, and Lin  \cite{CKL09} for a related nonvariational system and its parabolic form. These works established the regularity of the free boundary: near points where only two phases meet, the interface is real-analytic. Singular points may occur when three or more phases meet, but they belong to a smaller closed set of Hausdorff dimension at most $n-2$.

The time-dependent theory is less developed. Dancer, Wang, and Zhang \cite{DWZ11,DWZ12} obtained uniform estimates in the competition parameter and identified differential inequalities satisfied by parabolic segregated limits. Audrito, Serra, and Tilli \cite{AST21} later gave a variational construction of global segregated weak solutions for parabolic reaction--diffusion systems. Recent work also shows that systems with three or more populations may exhibit more complicated time-dependent patterns, including rotating spiral interfaces; see \cite{STVZ25}.

In recent years, there has been growing interest in studying equations with fully nonlinear diffusion as they arise in many applications; see \cite{CIL92,CC95}.  
In the context of segregation problems, Quitalo \cite{Q13} introduced a stationary model for the Gause--Lotka--Volterra model of multiple competing species, governed by the \textit{Pucci extremal operators}, see \eqref{eq:pucci}, and proved existence of viscosity solutions and convergence to a segregated limit. Caffarelli, Patrizi, Quitalo, and Torres \cite{CPQT19} studied the limiting free boundary problem and developed the free boundary regularity. They considered the problem for two species and allowed more general fully nonlinear operators. More precisely,
\begin{equation*} 
\begin{cases}
     F^-(D^2 u)  = 0 & \text{ in } \{u>0\},\\
F^+(D^2 u)= 0 & \text{ in } \{u<0\},\\
F^-(D^2 u)  \leq 0 & \text{ in } \Omega,\\
 F^+(D^2 u)   \geq 0 & \text{ in } \Omega,\\
 u = f & \text{ on }  \partial \Omega,
\end{cases}
\end{equation*}
where $F^+$ and $F^-$ satisfy suitable properties  (see Section~\ref{sec:operators}). This problem is the most relevant for this work, as it is the elliptic counterpart of \eqref{eq:FBP}. 

Other related models with long-range interaction have also been considered: with linear diffusion \cite{CPQ17} and, more recently, with fully nonlinear diffusion \cite{CPT26}. In contrast to the previous works, the limiting populations remain separated by a positive distance.

Two-phase \textit{anisotropic} free boundary problems also arise in other settings; see, for instance, \cite{AM13,AT15,KLS17,ST23,GS25} for equations in divergence form, and \cite{CDSS18,BDS25,PSw22,PS23} for fully nonlinear equations. Most of the existing literature concerns elliptic problems, whereas the parabolic theory has received less attention.

\subsection{Properties of $F^+$ and $F^-$} \label{sec:operators}
Let $\S_n$ be the set of $n\times n$ symmetric matrices. 
Given $0<\lambda \leq \Lambda$, we define the set of uniformly elliptic matrices as
$$
\mathcal{A}_{\lambda,\Lambda} := \big\{ A\in \S_n : \lambda I \leq A \leq \Lambda I \big\}.
$$
Recall the Pucci extremal operators are given by
\begin{equation} \label{eq:pucci}
\M^-(M) := \inf_{A\in \mathcal{A}_{\lambda,\Lambda}} \text{tr}(AM)
\quad \text{ and } \quad
\M^+(M) := \sup_{A\in \mathcal{A}_{\lambda,\Lambda}} \text{tr}(AM).
\end{equation}
In the following, we will assume that $F^\pm : \S_n \to \R$ satisfy three structural conditions:\medskip

\begin{enumerate}[$(i)$]
    \item \textit{Uniform ellipticity:} For any $M,N\in \S_n$, with $N\geq 0$,
    $$\M^-(N)\leq F^{\pm}(M+N)-F^{\pm}(M)\leq \M^+(N).$$
    \item \textit{Positive 1-homogeneity:} For any $M \in \S_n$ and $r>0$,
    $$F^{\pm}(r M)=r F^{\pm}(M).$$
    \item \textit{Ordering:} For any $M \in \S_n$,
    $$F^-(M)\leq \tr(M)\leq F^+(M).$$
\end{enumerate}

\begin{rem}
    From $(ii)$, we have $F^\pm(0)=0$. Moreover, we may write $(i)$ equivalently as
    $$
    \M^-(M-N)\leq F^{\pm}(M)-F^{\pm}(N)\leq \M^+(M-N), \quad \text{ for all } M,N\in \S_n. 
    $$
    Taking $N=0$, and using $(iii)$, we get
    \begin{equation} \label{eq:ordering}
    \M^-(M)\leq F^-(M)\leq F^+(M)\leq \M^+(M).
    \end{equation}
    We point out that the ordering assumption involving the trace operator is only needed in the blow-up analysis of the free boundary (see Section~\ref{s:blowup}), while in the rest of our arguments we only use $F^-(M) \leq F^+(M)$.
\end{rem}

\begin{rem}[Examples]
Clearly, the Pucci operators satisfy the above conditions, provided $\lambda \leq 1 \leq \Lambda$, so that $(iii)$ holds. In fact, these properties hold for more general concave or convex operators. Indeed, let $\mathcal{A}^-,\mathcal{A}^+ \subseteq \mathcal{A}_{\lambda,\Lambda}$ and consider
$$
F^-(M) = \inf_{A\in \mathcal{A}^-} \tr(AM), \qquad F^+(M)= \sup_{A\in \mathcal{A}^+} \tr(AM).
$$
Then, $F^-$ is concave, $F^+$ is convex, and they satisfy $(i)$ and $(ii)$. Conversely, any concave (resp. convex) operator satisfying $(i)$ and $(ii)$ can be written in this form, and in particular, if $F$ is differentiable away from $0$, the matrices $A$ are the range of $DF$. Moreover, $(iii)$ holds~if 
$$
I \in \mathcal{A^-}\cap \mathcal{A^+}.
$$
Such operators have been considered in \cite{CPQT19}. 

In this paper, we do not assume convexity or concavity of the operators. We consider a more general class satisfying only the conditions $(i)$--$(iii)$.
\end{rem}

\subsection{Main results and strategy}
We now state the main results and briefly discuss the strategy of the proof. First, we construct a Lipschitz viscosity solution.

\begin{thm}[Existence and Lipschitz regularity] \label{thm:main1}
Let $f: \overline\Omega \to \R$ and $g: \partial \Omega\times [0,T]\to \R$ be continuous functions with $f^\pm \not\equiv 0$ and $g^\pm\not\equiv0$.
Assume that $f=g$ on $\partial\Omega\times\{0\}$. Then:
\begin{enumerate}[$(i)$]
    \item There exists a viscosity solution $u$ of \eqref{eq:FBP} such that $u^\pm \not\equiv 0$. 
    \item If $f,g \in C^{1,\alpha}$, for some $\alpha \in (0,1)$, then the solution $u$ given in $(i)$ is Lipschitz continuous (in the parabolic sense) on $\overline\Omega\times [0,T]$, with
 $$
 \|u\|_{C^{0,1}(\overline\Omega\times [0,T])} \leq C \big ( \|f\|_{C^{1,\alpha}(\overline\Omega)} + \|g\|_{C^{1,\alpha}(\partial\Omega\times[0,T])}\big),
 $$
 where $C>0$ depends only on $n$, $\lambda$, $\Lambda$, $\alpha$, $\Omega$, and $T$. 
\end{enumerate}
\end{thm}

Note that, in general, the boundary regularity in $(ii)$ is false if $f,g$ are merely Lipschitz (see Remark~\ref{rem:data}), or even $C^1$.

The idea to prove $(i)$ is to regularize the jump of the operators across the free boundary. We consider a smooth interpolation $G_\vep$ between $F^-$ and $F^+$ and solve the regularized problem 
\begin{equation}\label{eq:reg}
    \partial_t u^\vep - G_\vep(D^2 u^\vep , u^\vep)=0 \quad \text{ in } \Omega\times(0,T],
\end{equation}
 by Perron's method (see Theorem~\ref{thm:existencereg}). The ellipticity and monotonicity of $G_\vep$ ensure that the comparison principle holds for \eqref{eq:reg}. 
 
 To prove $(ii)$, we show that the solutions $u^\vep$ to \eqref{eq:reg} satisfy the uniform Lipschitz estimate
 \begin{equation} \label{eq:uniflip}
     \|u^\vep\|_{C^{0,1}(\overline\Omega\times[0,T])} \leq C_0,
 \end{equation}
where $C_0$ is independent of $\vep$ (see Theorem~\ref{thm:uniformest}). The Lipschitz norm is understood in the parabolic sense, i.e., $u^\vep$ is $C^{0,1}$ in space and $C^{1/2}$ in time (see Subsection~\ref{sec:parabolicspaces}). The uniform $C^{0,1}$ estimate in space follows by first obtaining boundary estimates through Pucci barriers and then propagating them into the interior. 
The uniform $C^{1/2}$ estimate in time follows by comparing $u^\vep$ with a suitable global solution to the heat equation (see Appendix~\ref{a:heateqnsol}). 
Once \eqref{eq:uniflip} is established, passing to the limit ($\vep\to 0$) in \eqref{eq:reg}, we obtain a Lipschitz viscosity solution to the free boundary problem \eqref{eq:FBP}, thus proving Theorem~\ref{thm:main1}. This approach avoids the fixed-point construction and the envelope arguments used in \cite{CPQT19}.\medskip

Regarding the regularity of $u$, we note that, in the interior of each phase, $u$ satisfies a uniformly parabolic fully nonlinear equation, so classical regularity theory gives 
$$u\in C_{\rm loc}^{1,\alpha}(\Omega\times(0,T]\setminus \Gamma), \quad \text{ for some } \alpha\in (0,1),$$
in the parabolic sense (see Subsection~\ref{sec:parabolicspaces}).
Hence, we only need to study the behavior of $u$ across $\Gamma$. To this end, fix $(x_0,t_0)\in \Gamma$, and for $r>0$, consider the rescaled function
$$
u_r(x,t) = \frac{1}{r}u(rx+ x_0, r^2 t + t_0).
$$
Since $u$ is Lipschitz and $u(x_0,t_0)=0$, the family $\{u_r\}$ is Lipschitz, uniformly in $r$. Hence, by compactness, along a subsequence, 
$$
u_{r_j} \to u_0 \quad \text{ locally uniformly as $r_j\to 0$}. 
$$
The key step is to classify all possible blow-up limits $u_0$. If $u$ vanishes faster than linearly at $(x_0,t_0)$, in the parabolic scaling, then every blow-up limit is identically zero. This motivates the following definition.

\begin{defn}\label{def:regpts}
    We say that $(x_0,t_0)\in \Gamma$ is a \textit{regular point} of the free boundary if there exist positive constants $r_0=r_0(x_0,t_0)$ and $M=M(x_0,t_0)$ such that 
$$
\sup_{Q_r(x_0,t_0)} |u(x,t)|\geq Mr, \qquad \text{for all } 0<r<r_0.
$$
Otherwise, we say that $(x_0,t_0)$ is a \textit{singular point.}
\end{defn}
 The definition of a regular point only states that $u$ grows at least linearly from one side at every sufficiently small scale. We then prove in Lemma~\ref{lem:nondegeneracy} a non-degeneracy result showing that linear growth in one phase implies linear growth in both phases. Namely,
 $$
N\, r \leq \sup_{Q_r(x_0,t_0)} u^\pm \leq \sqrt{2} L\, r, \qquad \text{ for all } 0<r<r_0.
$$
 The upper bound follows from the Lipschitz continuity of $u$ and shows that the linear growth at $(x_0,t_0)$ is optimal.
We arrive at the following dichotomy:

\begin{thm}[Differentiability in space across the free boundary] \label{thm:main2}
Let \(u\) be a Lipschitz viscosity solution of \eqref{eq:FBP}. Then $u$ is differentiable with respect to $x$ on $\Gamma$. 
Moreover:
\begin{enumerate}
    \item[$(i)$] If $(x_0,t_0)\in\Gamma$ is a singular point, then $$\nabla u(x_0,t_0)=0.$$
    \item[$(ii)$] If $(x_0,t_0)\in\Gamma$ is a regular point, then $u$ has a parabolic $C^{1,\alpha}$ expansion, with $$\nabla u(x_0,t_0)\neq0.$$
\end{enumerate}
 \end{thm}

 In particular, spatial differentiability holds at every free boundary point, regardless of the regularity of the interface. 
To see this, assume that $(0,0)\in \Gamma$. The theorem is proved in two main steps. First, we classify all possible blow-up limits $u_0$ (see Lemma~\ref{lem:blowups}), using a parabolic ACF-type monotonicity formula (see Theorem~\ref{thm:ACFlocal}). We may use this tool thanks to the ordering assumption $(iii)$ (see Remark~\ref{rem:ordering}). 
A technical difficulty lies in passing to the limit in the monotonicity formula, whose integrals involve the gradients of the rescaled solutions, so uniform convergence to a blow-up limit does not justify this. Since different operators govern the two phases, interior regularity estimates control the gradients only away from the free boundary. The main challenge is to control the energy near the interface.

If $(0,0)$ is a regular point, we see that $u_0$ must be a two-plane profile, i.e.,
$$
 u_{r_j} \to u_0(x)=a (x\cdot e)^+ - b (x \cdot e)^-, \quad \text{ for some }  a, b>0, \quad |e|=1.
$$
By the global viscosity inequalities, we deduce that the slopes must be equal: $a=b$. Hence, for sufficiently small scales, $u$ is close to a linear function $a(x\cdot e)$. Now, a priori, the slope $a$ and the direction $e$ could depend on the subsequence $r_j$. The second step consists of proving this is not the case, that is, the blow-up is unique. Indeed, after rescaling, we may assume that $a=1$ and that the graph of $u$ is $\delta$-flat in the $e$-direction at scale 1, i.e.,
$$
\sup_{Q_1} |u-x\cdot e|\leq \delta
$$
(see Remark~\ref{rem:normal}). Inspired by the work of Daniela De Silva in \cite{DS11}, we argue that if this condition holds, then the graph of $u$ must be $\delta/2$-flat in the $\nu$-direction at scale $r$, with $\nu$ close to $e$ (see Theorem~\ref{thm:IF}).
An iteration of this improvement of flatness gives the $C^{1,\alpha}$ regularity of $u$ at $(0,0)$, with quantitative estimates. Note that this implies the blow-up limit is uniquely determined by the linear asymptotic expansion of $u$ at $(0,0)$.

Now, if $(0,0)$ is a singular point, then $u_0$ is identically zero, and this is true for any convergent subsequence $u_{r_j}$. Hence, the blow-up limit is unique, and the gradient of $u$ is well-defined at $(0,0)$. In particular, $\nabla u(0,0)=0$, thus concluding Theorem~\ref{thm:main2}.

\begin{rem} 
Our above strategy differs from the one used in \cite{CPQT19} for the elliptic problem. There, the authors classify blow-up limits at regular points and establish an asymptotic expansion at points admitting a tangent ball from one side. Then they prove the $C^{1,\alpha}$ regularity of the free boundary using the classical approach developed by Caffarelli \cite{C87,C89}, based on sup-convolutions and improvement of monotonicity cones. 
\end{rem}

Finally, we deduce some results about the set of regular points. 

\begin{thm}[Free boundary regularity] \label{thm:main3}
Let $u$ be a Lipschitz viscosity solution of \eqref{eq:FBP}, and let $\mathcal{R}$ be the set of regular points. Then $\mathcal{R}$ is an open subset of $\Gamma$ and locally a $C^{1,\alpha}$ hypersurface, in the parabolic sense, for some $0<\alpha<1$. 
\end{thm}

The proof follows from $C^{1,\alpha}$ estimates at regular points by standard arguments. 

\begin{rem}[Open problem]
The singular set remains to be understood. Even for the isotropic case $F=F^+=F^-$, it is not known whether the critical nodal set $\{u=0, \nabla u=0\}$ of solutions to $F(D^2 u)=0$  or $\partial_t u - F(D^2 u)=0$ is lower dimensional, in general. This is a challenging open problem as classical tools for linear equations, such as Almgren-type monotonicity formulas, are not available in this setting. We will investigate this issue in a future work.
\end{rem}

\subsection{Organization of the paper} 
In Section~\ref{s:prelim}, we introduce notation, definitions, and tools used throughout the paper. In Section~\ref{s:existreg}, we construct viscosity solutions via a regularization argument and establish the Lipschitz regularity. Section~\ref{s:nondegen} is devoted to the non-degeneracy at regular points. In Section~\ref{s:blowup}, we classify blow-up limits and deduce spatial differentiability at singular points. In Section~\ref{sec:IF}, we prove the improvement of flatness, which we iterate in Section~\ref{sec:regularity} to obtain the $C^{1,\alpha}$ regularity of the solution and the free boundary at regular points. Finally, Appendices~\ref{a:flattrans} and \ref{a:heateqnsol} contain auxiliary results on flat transmission problems and global solutions to the heat equation, respectively.

\section{Preliminaries}\label{s:prelim}

We introduce the notation, definitions, and basic tools that we will use in the paper.

\subsection{Notation} 

Given $(x,t)\in \R^{n+1}$, we call $x\in \Rn$ the space variable and $t \in \R$  the time variable. 
Given a function $u(x,t)$, we denote by $\nabla u$ the gradient of $u$ with respect to $x$, $D^2 u$ the Hessian with respect to $x$, and $\partial_t u$ the partial derivative with respect to $t$.  We denote $u^+:=\max\{u,0\}$ and $u^-:= -\min\{u,0\}$, the positive and negative parts of $u$, respectively.

 We call a \textit{parabolic cylinder} any set of the form 
 $$Q=\Omega\times (t_1, t_2],$$ 
 where $\Omega$ is a smooth bounded domain in $\Rn$ and $t_1<t_2$. We denote by $\partial_p Q$ the \textit{parabolic boundary} of $Q$, i.e., $\partial_p Q= \partial_b Q \cup \partial_l Q$, where
 $ \partial_b Q= \Omega \times \{t_1\}$ is the \textit{bottom}, and
 $\partial_l Q= \partial \Omega \times [t_1,t_2]$ is the \textit{lateral} part of the boundary.
  
Given $(x,t), (y,s)\in \R^{n+1}$, we define the \textit{parabolic distance} as
$$
d_p((x,t), (y,s)) :=\big(|x-y|^2 + |t-s|\big)^{1/2}.
$$
 We denote Euclidean balls and cylinders in the usual way. Let $B_r = \{x \in \Rn : |x|< r\}$ be the ball of radius $r$ centered at $0$ in $\R^n$. For $x_0\in \Rn$ and $t_0 \in \R$, we write: 
 \begin{align*}
   B_r(x_0) &= B_r + x_0,\\
  Q_r(x_0,t_0) &= B_r(x_0)\times(t_0-r^2,t_0].
 \end{align*}

 \subsection{Parabolic spaces} \label{sec:parabolicspaces}

 Let $C(Q)$ be the set of continuous functions in $Q$. 
Given $k\in \N$, $C^{k}(Q)$ is the set of all functions $u$ such that $D_x^m D_t^l u \in C(Q)$ for all $|m|+2 l\leq k$,  where $m$ is a multi-index and $l\geq 0$ is an integer.
 
 Given $\alpha\in(0,1]$, we define $C^{0,\alpha}(\overline{Q})$ as the set of all functions $u\in C(\overline{Q})$ such that
 $$
\|u\|_{C^{0,\alpha}(\overline{Q})} := \|u\|_{L^\infty(Q)} + [u]_{C^{0,\alpha}(\overline{Q})} < \infty,
 $$
 where the $\alpha$-H\"{o}lder seminorm (or Lipschitz seminorm if $\alpha=1$) of $u$ is given by
 \begin{align*}
 [u]_{C^{0,\alpha}(\overline{Q})} := \sup_{ (x,t)\neq (y,s)} \frac{|u(x,t)-u(y,s)|}{d_p((x,t),(y,s))^{\alpha}}.
 \end{align*}
 We also define the $\alpha$-H\"{o}lder seminorm in $t$ as
 $$
[u]_{C_t^{0,\alpha}(\overline{Q})} := \sup_{ (x,t)\neq (x,s)} \frac{|u(x,t)-u(x,s)|}{|t-s|^\alpha}.
 $$
Note that if $u\in C^{0,\alpha}(\overline Q)$, then $u$ will be $\alpha$-H\"{o}lder continuous in $x$ and $\alpha/2$-H\"{o}lder continuous in $t$. 
Moreover, we define $C^{1,\alpha}(\overline{Q})$ as the set of all functions $u\in C(\overline{Q})$ such that
\begin{align*}
\|u\|_{C^{1,\alpha}(\overline{Q})}&:= \|u\|_{L^\infty(Q)} + \|\nabla u\|_{C^{0,\alpha}(\overline Q)}+ [u]_{C_t^{0,(1+\alpha)/2}(\overline{Q})} <\infty.
\end{align*}

 \subsection{Viscosity solutions and the Pucci class} 
Let $Q\subset \R^{n+1}$ be an open set. 

 \begin{defn}
 Given $u: Q\to \R$, we say that $\varphi\in C^2$ touches $u$ from above (resp. below) at $(x_0,t_0)\in Q$ if $u(x_0,t_0)=\varphi(x_0,t_0)$, and there is $r>0$ such that
 $$u(x,t) \leq \varphi(x,t) \quad (\hbox{resp.}~\geq) \quad \hbox{for all}~ (x,t)\in Q_r(x_0,t_0) \subseteq Q.$$
 \end{defn}

 \begin{defn} 
 We say that $u\in C(Q)$ is a viscosity subsolution (resp.~supersolution) to 
 $$\partial_t u - F(D^2 u)=0\quad  \text{ in } Q$$ 
 if whenever $\varphi\in C^2$ touches $u$ from above (resp.~below) at $(x_0,t_0)\in Q$, it holds that
 $$\partial_t \varphi(x_0,t_0) - F(D^2\varphi (x_0,t_0)) \leq 0 \quad (\hbox{resp.}~\geq).$$
We call $u$ a \textit{viscosity solution} if it is both a viscosity subsolution and a viscosity supersolution.
 \end{defn}

 \begin{defn}
     Given $f\in C(Q)$, we say that $u\in C(Q)$ belongs to the Pucci class in $Q$, and write, $u \in S(f)$ in $Q$, if $u$ satisfies
     $$
    \M^-(D^2 u) + f \leq \partial_t u \leq \M^+(D^2 u) + f 
     $$
     in the viscosity sense.
 \end{defn}

 \begin{rem}
     If $u$ is a viscosity solution to \eqref{eq:FBP}, then by \eqref{eq:ordering},
     $$
    \M^-(D^2 u) \leq  F^-(D^2 u) \leq \partial_t u \leq F^+(D^2 u)  \leq \M^+(D^2 u),
     $$
     in the viscosity sense.
     Hence, $u\in S(0)$ in $\Omega\times(0,T]$. In particular, the regularity theory for functions in the Pucci class holds. More precisely, $u$ is locally H\"older continuous in $\Omega\times(0,T]$; see \cite[Theorem~4.19]{WI92}.
 \end{rem}

\subsection{A parabolic monotonicity formula}

Monotonicity formulas are fundamental tools in the study of free boundary regularity. A celebrated formula for elliptic two-phase problems was introduced by Alt, Caffarelli, and Friedman, and was later extended to parabolic problems by Caffarelli and Kenig. For a detailed discussion, see \cite[Section~12.4]{CS05}. 

Next, we recall the parabolic formula.
For $x\in \Rn$ and $t>0$, the heat kernel is given by
$$
G(x,t)= \frac{1}{(4\pi t)^{n/2}} e^{-\frac{|x|^2}{4t}}.
$$
For a function $v$ and any $t>0$, we define
$$
I(t; v):= \frac{1}{t} \int_{-t}^0 \int_{\Rn} |\nabla v(x,s)|^2 G(x,-s)\,dx ds.
$$

\begin{thm}[The global formula] \label{thm:ACFglobal}
    Let $u_1$ and $u_2$ be two nonnegative continuous functions defined on the strip $\R^n \times [-1,0]$, and satisfying the following properties:
    \begin{enumerate}
        \item[$(i)$] $\partial_t u_i -\Delta u_i \leq 0$, for $i=1,2$;
        \item[$(ii)$] $u_1u_2\equiv0$;
        \item[$(iii)$] $u_1(0,0)=u_2(0,0)=0$.
    \end{enumerate}
    Assume that the $u_i$ have moderate growth at infinity, say
    $$
    \int_{B_R} u_i^2(x,-1)\, dx \leq C e^\frac{R^2}{(4+\vep)},
    $$
    for $R$ large and some $\vep>0$.
    Then for $0<t\leq 1$, the function
    $$
    J(t ; u_1,u_2) :=I(t; u_1) I(t; u_2)
    $$
    is monotone increasing in $t$.
\end{thm}

\begin{rem} \label{rem:scaling}
    If $v_r(x,t)=r^{-1} v(rx,r^2t)$, then $J(t; (u_1)_r, (u_2)_r) = J(r^2 t; u_1, u_2)$.
\end{rem}

We will also need the case of equality in the monotonicity formula. For a proof, see for instance \cite[Proposition 5.3]{AS16} (see also \cite[Lemma 12.9]{CS05}).

\begin{prop} \label{prop:rigidity}
    Let $u_1$ and $u_2$ satisfy the assumptions of Theorem~~\ref{thm:ACFglobal}. Then $J(t; u_1, u_2)$ is a strictly positive constant if and only if the $u_i$ are two complementary linear functions, i.e., there are constants $a,b >0$ and a unit vector $e$ such that $u_1=a (x\cdot e)^+$ and $u_2=b (x\cdot e)^-$.
\end{prop}

Finally, we need the localized version of the global monotonicity formula.

\begin{thm}[The local formula] \label{thm:ACFlocal}
    Assume that $u_1,u_2\in L^2(Q_1)$ satisfy the assumptions of Theorem~\ref{thm:ACFglobal} in $Q_1$. Let $\psi$ be a smooth cut-off function in $x$, i.e., $\psi\equiv 0$ outside $B_{2/3}$ and $\psi\equiv1$ in  $B_{1/2}$. Then if $w_i= u_i \psi$ and $J(t)=J(t; w_1,w_2)$,
    $$
    J(0^+) - J(t) \leq A e^{-C/t} \|u_1\|_{L^2(Q_1)}^2 \|u_2\|_{L^2(Q_1)}^2
    $$
    for some $A,C>0$, depending only on $n$.
\end{thm}

\begin{rem}  \label{rem:ordering}
In this work, the monotonicity formula will be fundamental for identifying blow-up limits. This is the key point where we use the ordering assumption
$$
F^-(M) \leq \tr(M)\leq F^+(M).
$$
Indeed, let $u$ be a viscosity solution to \eqref{eq:FBP}, and set $u_1=u^+$ and $u_2=u^-$. Then
$$
\partial_t u_1 = F^-(D^2 u_1) \leq \Delta u_1 \quad \text{ in } \{u_1>0\},
$$
and 
$$
\partial_t u_2 = - F^+(-D^2 u_2) \leq \Delta u_2 \quad \text{ in } \{u_2>0\},
$$
in the viscosity sense. In particular, since $u_1$ and $u_2$ are nonnegative and vanish continuously outside their respective phases, these inequalities extend across their zero sets, in the viscosity sense, and hence also in the distributional sense.
Therefore, $u_1$ and $u_2$ are subcaloric, and the parabolic monotonicity formula applies to this pair of functions.
\end{rem}

\section{Existence and Lipschitz regularity}\label{s:existreg}

We will prove existence of a (Lipschitz) viscosity solution of \eqref{eq:FBP} by first regularizing the problem, proving uniform estimates, and then passing to the limit. 

\subsection{Regularized problem} For $\vep>0$, let $h_\vep: \R\to [0,1]$ be a smooth approximation of $\chi_{[0, \infty)}$ satisfying $h_\vep(r) = \chi_{[0,\infty)}(r)$ for all $|r| > \vep$, and $h_\vep'\geq 0$. Consider the regularized operator $G_\vep : \S_n \times \R \to \R$ given by
$$
G_\vep(M, r ):= h_\vep(r) F^-(M) + (1-h_\vep(r)) F^+(M) \quad \text{ for }  M\in \S_n, \ r\in \R,
$$
and the corresponding $\vep$-problem:
\begin{equation} \label{eq:regularized}
\begin{cases}
    \partial_t u^\vep - G_\vep(D^2 u^\vep, u^\vep)=0 & \text{ in } \Omega \times (0,T],\\
    u^\vep(x,0) = f(x)  & \text{ in } \Omega,\\
    u^\vep(x,t) =g(x,t) & \text{ on } \partial \Omega \times (0,T].
    \end{cases}
\end{equation}

We list some useful properties satisfied by $G_\vep$. The proof is simple so we omit it here.
\begin{lem}[Properties of $G_\vep$] \label{lem:propG}
The following properties hold:
\begin{enumerate}[$(i)$]
\item Uniform ellipticity: For any $M,N\in \S_n$, with $N\geq 0$, and $r\in \R$,
    $$\M^-(N)\leq G_\varepsilon(M+N,r)-G_{\varepsilon}(M,r)\leq \M^+(N).$$
\item Monotonicity in $r$:  For any $M\in \S_n$ and $r\leq s$,
$$
G_\vep(M,r) \geq G_\vep(M, s).
$$
\item Ordering: For any $M\in \S_n$ and $r\in \R$,
$$
    F^-(M) \leq G_\vep(M,r) \leq F^+(M).
    $$
\end{enumerate}
\end{lem}

 We will show that there is a viscosity solution $u^\vep$ with Lipschitz norm independent of $\vep.$ To prove existence, we only require the boundary data to be continuous. 

For simplicity, write
$$
Q_T := \Omega \times (0,T].
$$

\begin{thm}[Existence $\vep$-problem] \label{thm:existencereg}
    Let $\vep>0$. Let $f: \overline\Omega \to \R$ and $g: \partial \Omega\times [0,T]\to \R$ be continuous functions satisfying $f(x)=g(x,0)$, for all $x\in \partial \Omega$.
    Then there exists a viscosity solution $u^\vep \in C(\overline Q_T)$ of the $\vep$-problem \eqref{eq:regularized}. In particular,
    $$
     F^-(D^2 u^\vep) \leq  \partial_t u^\vep \leq  F^+(D^2u^\vep) \quad \text{ in } Q_T,
    $$
    in the viscosity sense.

\end{thm}

We need the following comparison principle given in \cite[Theorem~8.2]{CIL92}; see also \cite[Theorem~IV.1]{IL90}. For the sake of clarity, we state it in the form we will use it.

\begin{lem}[Comparison principle] \label{lem:comparison}
    Let $G : \S_n \times  \R \to \R $ be a continuous operator satisfying:
     \begin{enumerate}[$(i)$]
     \item $G$ is degenerate elliptic: For all $M,N\in \S_n$ with $M\leq N$ and $r\in \R$,
     $$G(M,r) \leq G(N,r).$$
     \item $G$ is proper: For all $M\in \S_n$ and $r\leq s$,
     $$G(M,r)\geq G(M,s).$$
     \end{enumerate}         
        Let $\Omega$ be a bounded open set in $\Rn$ and $T>0$. Consider the equation
        \begin{equation} \label{eq:comparison} 
    \partial_t w - G(D^2 w, w)=0 \quad \text{ in } Q_T:=\Omega\times (0,T].
\end{equation}
        Assume that $u$ is a viscosity subsolution to \eqref{eq:comparison} and $v$ is a viscosity supersolution to \eqref{eq:comparison}. If $u \leq v$ on $\partial_p Q_T$, then $u\leq v$ in $Q_T$.
\end{lem}

\begin{rem}
The operator $G_\vep$ satisfies properties $(i)$ and $(ii)$. Hence, the comparison principle applies to the regularized equation in \eqref{eq:regularized}.
\end{rem}

\begin{proof}[Proof of Theorem~\ref{thm:existencereg}]
    Let $\vep>0$. First, we construct a suitable subsolution and supersolution to \eqref{eq:regularized}. Consider the solutions to the Cauchy-Dirichlet problems
        \begin{equation} \label{eq:subsol}
           \begin{cases}
        \partial_t \underline u - \M^-(D^2 \underline u) =0 & \text{ in } \Omega\times (0,T],\\
        \underline u(x,0)=f(x) & \text{ in }  \Omega,\\
        \underline u(x,t)= g(x,t) & \text{ on } \partial \Omega \times(0,T],
    \end{cases}
    \end{equation}
    and
    \begin{equation} \label{eq:supersol}
            \begin{cases}
        \partial_t \bar u - \M^+(D^2 \bar u) =0 & \text{ in } \Omega\times (0,T],\\
        \bar u(x,0)=f(x) & \text{ in }  \Omega,\\
        \bar u(x,t)= g(x,t) & \text{ on } \partial \Omega \times(0,T].
    \end{cases}
    \end{equation}
    By \eqref{eq:ordering} and Lemma~\ref{lem:propG}, $(iii)$, we have
    $$
     \partial_t \underline u = \M^-(D^2 \underline u)  \leq F^-(D^2 \underline u)\leq G_\vep(D^2 \underline u, \underline u) \quad \text{ in } Q_T.
    $$
    Hence, $\underline u$ is a subsolution to \eqref{eq:regularized}. Similarly, $\bar u$ is a supersolution to \eqref{eq:regularized}.
   Since the comparison principle holds for \eqref{eq:regularized}, by a standard application of Perron's method, there is a continuous viscosity solution $u^\vep$ to \eqref{eq:regularized}. 
    Moreover,
    $$
     F^-(D^2 u^\vep) \leq  \partial_t u^\vep \leq  F^+(D^2u^\vep) \quad \text{ in } Q_T,
    $$
    in the viscosity sense. 
    \end{proof}

 \begin{thm}[Uniform Lipschitz estimate] \label{thm:uniformest}
    Let $\vep>0$ and $L_0>0$. Assume that $f: \overline\Omega \to \R$ and $g: \partial \Omega\times [0,T]\to \R$ are $C^{1,\alpha}$-functions, for some $\alpha \in (0,1)$, satisfying
    $$
    \|f\|_{C^{1,\alpha}(\overline\Omega)} + \|g\|_{C^{1,\alpha}(\partial\Omega\times[0,T])}\leq L_0.
    $$
     If $u^\vep$ is a viscosity solution of \eqref{eq:regularized}, then $u^\vep\in C^{0,1}(\overline\Omega\times[0,T])$,  and
    $$
    \|u^\vep\|_{C^{0,1}(\overline\Omega\times[0,T])} \leq C_0,
    $$
    where $C_0>0$ depends only on $n$, $\lambda$, $\Lambda$, $\alpha$, $L_0$, $\Omega$, and $T$.
\end{thm}

\begin{rem} \label{rem:data}
The $C^{1,\alpha}$ regularity assumption on the boundary data in Theorem \ref{thm:uniformest} cannot be relaxed to merely Lipschitz. Indeed, even for harmonic functions, Lipschitz boundary data does not guarantee Lipschitz regularity up to the boundary.  For instance, let $\eta\in C_c^\infty(\R)$ be a cut-off function with $\eta\equiv1$ on $[-1,1]$, and 
consider in $\R^2_+ = \{y>0\}$ the harmonic extension of the Lipschitz function $f(x) = |x|\eta(x)$,
\[
u(x,y) = \frac{1}{\pi}\int_\R \frac{y\,|t| \eta(t)}{(x-t)^2+y^2}\,\mathrm{d}t.
\]
Then $u=f$ on $\{y=0\}$, but $u\notin C^{0,1}(\overline{\R^2_+})$, since evaluating at $x=0$ and substituting $s = t^2$,
\[
\partial_y u(0,y) \sim \int_0^1 \frac{s - y^2}{(s+y^2)^2}\,\mathrm{d}s = \int_0^{1/y^2} \frac{\tau-1}{(\tau+1)^2}\,\mathrm{d}\tau \sim \log(1/y^2) \to \infty.
\]
\end{rem}

\begin{proof}[Proof of Theorem \ref{thm:uniformest}]
Let $\vep>0.$ First, observe that, by the maximum principle, 
\begin{equation*} 
    \|u^\vep\|_{L^{\infty}(\Omega\times (0,T])}\leq \|f\|_{L^{\infty}(\Omega)}+\|g\|_{L^{\infty}(\partial \Omega \times (0,T])}\leq L_0.\medskip
\end{equation*}

    We now divide the proof into two steps:\medskip

    \textbf{Step 1:} \textit{Lipschitz regularity in space.} Fix $t\in [0,T]$. We will see that, for any $x, y\in \overline \Omega$,
    \begin{equation} \label{eq:lipspace}
            |u^\vep(x,t)- u^\vep(y,t)| \leq C_0 |x-y|,
    \end{equation}   
    where $C_0>0$ does not depend on $\vep$. To see this, we distinguish two main cases:\medskip

    \underline{Case 1:} Let $x\in \overline\Omega$ and $y\in \partial \Omega$. Let $\bar u$ be the solution to \eqref{eq:supersol}. Note that
    $$
    \partial_t u^\vep - F^+(D^2 u^\vep) \leq 0 \leq  \partial_t \bar u - F^+(D^2 \bar u)  \quad \text{ in } Q_T,
    $$
    and $u^\vep= \bar u$ on $\partial_p Q_T$. By the comparison principle for $F^+$, we have $u^\vep \leq \bar u$ in $Q_T$. Moreover, we know that $\bar u \in C^{0,1}(\overline Q_T)$, with
    $$
    \|\bar u \|_{C^{0,1}(\overline Q_T)} \leq C(\|f\|_{C^{1,\alpha}(\overline\Omega)} + \|g\|_{C^{1,\alpha}(\partial\Omega\times[0,T])})\leq CL_0=: C_0,
    $$
    where $C_0>0$ depends only on $n$, $\lambda$, $\Lambda$, $\alpha$, $L_0$, $\Omega$, and $T$.
    Hence,
    $$
    u^\vep(x,t)-u^\vep(y,t) \leq \bar u(x,t) - \bar u(y,t) \leq [\bar u]_{C^{0,1}(\overline Q_T)} |x-y| \leq C_0 |x-y|.
    $$
    The lower bound can be proved similarly using the solution $\underline u$ to \eqref{eq:subsol}.\medskip

    \underline{Case 2:} Let $x,y \in \Omega$, and write $y=x+h$, for some $h\in \Rn$. Define
    $$
    \Omega_h:= \Omega\cap (\Omega-h)= \big\{ x\in \Omega : x+h \in \Omega \big\}.
    $$
    Let $C_0$ be the constant in Case~1. For $x\in \Omega_h$ and $t \in [0,T]$, consider the functions
    $$
    v(x,t):= u^\vep(x+h, t), \qquad w(x,t):= u^\vep(x,t) + C_0 |h|.
    $$
    We claim that 
    \begin{equation} \label{eq:orderedfunc}
    v\leq w\quad  \text{ on } \overline \Omega_h\times [0,T].
    \end{equation} 
    Indeed, since $G_\vep$ does not depend on $x$, 
        $$
    \partial_t v - G_\vep(D^2 v, v) =0 \quad \text{ in } \Omega_h \times (0,T],
    $$
    in the viscosity sense. Moreover, by Lemma~\ref{lem:propG}, $(ii)$,
    \begin{align*}
     \partial_t w - G_\vep(D^2 w, w) &=   \partial_t u^\vep - G_\vep(D^2 u^\vep, u^\vep + C_0 |h| )\\
     & \geq  \partial_t u^\vep -   G_\vep(D^2 u^\vep, u^\vep) \\
     & =0
     \quad \text{ in } \Omega_h \times (0,T],
    \end{align*}
    in the viscosity sense. Now, on the parabolic boundary:
 
 \begin{enumerate}[$(i)$]
 \item At $t=0$ and for $x\in \Omega_h$:
 \begin{align*}
 v(x,0) - w(x,0) & = u^\vep(x+h, 0) - u^\vep(x,0) - C_0 |h|\\
 & = f(x+h) - f(x) - C_0 |h|\\
 & \leq (L_0-C_0)|h| \leq 0,
 \end{align*}
because $[ f ]_{C^{0,1}(\Omega)}\leq L_0$ and we can choose $C_0>L_0$.\medskip
 
 \item For $x\in \partial \Omega_h$ and $0<t\leq T$: either $x\in \partial \Omega$ or $x+h \in \partial \Omega$. By Case~1,
 \begin{align*}
 v(x,t) - w(x,t) & = u^\vep(x+h, t) - u^\vep(x,t) - C_0 |h|\leq0.
 \end{align*}
 \end{enumerate}
  Therefore, $v\leq w$ on $\partial_p \big(\Omega_h \times (0,T]\big).$\medskip
  
The comparison principle (Lemma~\ref{lem:comparison}) yields \eqref{eq:orderedfunc}, thus proving the claim.
The lower bound follows similarly. Hence,
  $$
  |u^\vep(x,t)-u^\vep(y,t)|\leq C_0|x-y|.
  $$
  
     \medskip
    \textbf{Step 2:} \textit{$1/2$-H\"older regularity in time.} Fix $x\in \overline \Omega$. We will see that, for any $s,t \in [0,T]$,
    $$
    |u^\vep(x,s)- u^\vep(x,t)| \leq C_0 |s-t|^{1/2},
    $$
    where $C_0>0$ is independent of $\vep$.
    
     Let $(x,s),(x,t)\in Q_T$ with $s<t$. It is enough to prove the estimate when $t-s \leq 1$. 
     Fix $B$ to be chosen large.  For $(y,\tau)\in \Rn\times (s,\infty)$, let 
    $$
    \Phi(y,\tau):= \frac{1}{(4\pi\Lambda(\tau-s))^{n/2}} \int_{\Rn} e^{-\frac{|z-y|^2}{4\Lambda(\tau-s)}} \big(u^\vep(x,s) + B |x-z|\big)\, dz.
    $$
    Then $\Phi$ is a solution to the global problem
    $$
    \begin{cases}
        \partial_\tau \Phi - \Lambda \Delta \Phi =0 & \text{ in } \R^n \times (s,\infty),\\
        \Phi(y,s)= u^\vep(x,s) + B |x-y| & \text{ for } y \in \Rn.
    \end{cases}
    $$
    Define the rescaled function $\Phi_\Lambda(y,\tau):=\Phi(y,\tau/\Lambda)$. Then $\Phi_\Lambda$ satisfies the heat equation in $\Rn\times (\Lambda s, \infty)$.
    By Lemma~\ref{lem:globalsol} and Remark~\ref{rem:globalsol},
    $$
    \Phi_\Lambda(x,\tau)= u^\vep(x,s) + C(n) B (\tau- \Lambda s)^{1/2}, \quad \text{for all } \tau >\Lambda s,
    $$
    for some dimensional constant $C(n)>0$. Rescaling back, we get
    \begin{equation} \label{eq:estphi}
    \Phi(x,\tau)= u^\vep(x,s) + C(n) B \Lambda^{1/2} (\tau- s)^{1/2}, \quad \text{for all } \tau > s.
    \end{equation}
    
    Now, define the barrier
    $$
    \Psi(y,\tau) := \Phi(y,\tau) + B (\tau-s)^{1/2}.
    $$
    Note that $\Psi$ is a supersolution since
    $$
    \partial_\tau \Psi - \Lambda \Delta \Psi = \frac{B}{2}(\tau-s)^{-1/2} > 0 \quad \text{ in } \Omega \times (s,t).
    $$
    Moreover, $\Phi$ is convex since the initial data is convex, and thus,
    $$
    \partial_\tau \Psi - F^+(D^2 \Psi) \geq  \partial_\tau \Psi -  \Lambda \Delta \Psi > 0 \quad \text{ in } \Omega \times (s,t).
    $$
    We will see that $u^\vep\leq \Psi$ on $\partial_p(\Omega \times [s,t])$. Indeed, by \eqref{eq:lipspace}, we have 
    $$
    u^\vep(y,s) \leq u^\vep(x,s) + C_0 |x-y| \leq u^\vep(x,s) + B |x-y| = \Phi(y,s) = \Psi(y,s), \quad \text{ for all } y \in \Omega,
    $$
    choosing $B\geq C_0$. Hence, $u^\vep \leq \Psi$ on the bottom $\Omega \times \{s\}$. 
    
    To prove the bound on the lateral boundary, note that $\Phi$ is monotone increasing in time since $\partial_\tau \Phi = \Lambda \Delta \Phi \geq 0$. By assumption, the boundary data $g$ is $C^{1,\alpha}$ in the parabolic metric, for some $\alpha \in (0,1)$, which means $g$ is $C^{(1+\alpha)/2}$ in time. Then for any $y \in \partial\Omega$ and $\tau\in [s,t]$, 
    \begin{align*}
        \Psi(y,\tau) &\geq \Phi(y,s) + B(\tau-s)^{1/2}\\
        &= u^\vep(x,s) + B|x-y| + B(\tau-s)^{1/2}\\
        &\geq g(y,s) - C_0|x-y| + B|x-y| + B(\tau-s)^{1/2}\\
        &\geq g(y,\tau) - L_0(\tau-s)^{(1+\alpha)/2} - C_0|x-y| + B|x-y| + B (\tau-s)^{1/2}\\
        &\geq g(y,\tau) = u^\vep(y,\tau),
    \end{align*} 
    where the last inequality holds taking $B = \max\{C_0,L_0\}$. Therefore, $u^\vep \leq \Psi$ on $\partial\Omega \times [s,t]$. By the comparison principle for $F^+$, we get
    $$
    u^\vep \leq \Psi \quad \text{ in }  \Omega \times [s,t].
    $$
    Applying the estimate at the point $(x,t)$, and using \eqref{eq:estphi}, we see that
    $$
        u^\vep(x,t) \leq \Psi(x,t) = \Phi(x,t) + B(t-s)^{1/2} = u^\vep(x,s) + B \big(C(n) \Lambda^{1/2}  + 1\big) |t-s|^{1/2}.
    $$
    The lower bound follows similarly taking $-u^\vep$.
    \end{proof}

    \subsection{Proof of Theorem~\ref{thm:main1}} 
    By Theorem~\ref{thm:existencereg}, for any $\vep>0$, there exists a continuous viscosity solution $u^\vep$ to the $\vep$-problem \eqref{eq:regularized}, satisfying the viscosity inequalities
    $$
     F^-(D^2 u^\vep) \leq \partial_t u^\vep  \leq F^+(D^2u^\vep).
    $$
    Moreover, by Theorem~\ref{thm:uniformest}, we have 
    $$
    \|u^\vep\|_{C^{0,1}(\overline\Omega\times[0,T])} \leq C_0,
    $$
    where $C_0>0$ is independent of $\vep$. By compactness, up to a subsequence, $u^\vep\to u$ uniformly on $\overline\Omega\times[0,T]$, where $u$ is Lipschitz continuous on $\overline\Omega\times[0,T]$, and satisfies the original problem \eqref{eq:FBP} in the viscosity sense by \cite[Theorem 6.1]{CKS00}.

    Since $f^\pm\not\equiv 0$ and $g^\pm \not\equiv 0$, by the continuity of $u$ up to the boundary, we have $u^\pm \not\equiv0$.\qed

    \section{Non-degeneracy at regular points}\label{s:nondegen}
     
In this section, we prove the non-degeneracy result at regular points (see Definition~\ref{def:regpts}).

\begin{lem}\label{lem:nondegeneracy}
Let $u$ be a Lipschitz solution of \eqref{eq:FBP}, with Lipschitz constant $L$. Let $(x_0,t_0)\in \Gamma$ be a regular point. Then there exist $N>0$ and $r_0>0$, depending on $(x_0,t_0)$, such that 
$$
N\, r \leq \sup_{Q_r(x_0,t_0)} u^\pm \leq \sqrt{2} L\, r, \qquad \text{ for all } 0<r<r_0.
$$
\end{lem}

\begin{proof}
Let $u$ be a Lipschitz solution of \eqref{eq:FBP}, with Lipschitz constant $L$.
Let $(x_0,t_0)\in \Gamma$ be a regular point. By definition, there are $M>0$ and $r_0>0$, depending on $(x_0,t_0)$, such that
\begin{equation}\label{eq:regpt}
\sup_{Q_r(x_0,t_0)} |u|\geq Mr, \qquad \text{for all } 0<r<r_0.
\end{equation}
Since $u$ is Lipschitz and $u(x_0,t_0)=0$, we have
$$
\sup_{Q_r(x_0,t_0)} u^\pm \leq \sqrt{2} L r, \qquad \text{for all } 0<r<r_0.
$$
Hence, we only need to prove the lower bound. Assume, for the sake of contradiction, that the conclusion fails for $u^-$ or $u^+$. Without loss of generality, we may assume that it fails for $u^-$, i.e., for any $\vep < \frac{M}{2}$, there exists $0<\rho<\frac{r_0}{2}$, such that
$$
 \sup_{Q_{2\rho}(x_0,t_0)} u^- < \vep (2\rho)<M\rho.
$$
Since $|u|=\max\{u^+,u^-\}$, by \eqref{eq:regpt}, we must have
$$
 \sup_{Q_\rho(x_0,t_0)} u^+ \geq  M \rho.
$$
Hence, there is $(y,s)\in \overline Q_\rho(x_0,t_0)\cap\{u>0\}$, such that
$$
u^+(y,s) \geq M\rho.
$$

Now, for $\vep >0$, let 
$$v:= u + 2\vep \rho\quad \text{ in } Q_{2\rho}(x_0,t_0) \subset \Omega \times (0,T].$$ 
Then $v\geq 0$ in $Q_{2\rho}(x_0,t_0)$, $v(x_0,t_0)=2\vep\rho$, and $v(y,s)\geq M\rho.$
Moreover, $v$ satisfies
$$
 F^-(D^2 v)\leq  \partial_t v \leq F^+(D^2 v)  \quad \text{ in } Q_{2\rho}(x_0,t_0),
$$
in the viscosity sense, so $v\in S(0)$. We distinguish two cases:

\begin{enumerate}
    \item[$(i)$] If $t_0-\rho^2\leq s\leq t_0-\rho^2/2$, then by the parabolic Harnack inequality for the Pucci class, applied to $v$ in $Q_{2\rho}(x_0,t_0)$, there is some $c_0=c_0(n,\lambda,\Lambda)>0$ such that
    $$
    M\rho\leq v(y,s) \leq c_0 \inf_{Q_{\rho/2}(x_0,t_0)} v\leq c_0 v(x_0,t_0)= 2c_0 \vep \rho,
    $$
    which is a contradiction, taking $\vep < M/(2c_0)$.\medskip
    
    \item[$(ii)$] If $t_0-\rho^2/2< s\leq t_0$, then let $\tau := s -\delta^2\rho^2$, for some $0<\delta<1/2$ to be chosen small. Since $u$ is Lipschitz (in the parabolic sense), then so is $v$, with the same constant $L$, and
    $$
    v(y, \tau) \geq v(y,s) - L |\tau-s|^{1/2} \geq M \rho - L \delta \rho=(M-\delta L)\rho>\frac{M}{2}\rho,
    $$
    provided $\delta < M/(2L)$.
    By the parabolic Harnack inequality, there is $c_1(n,\lambda,\Lambda,\delta)>0$ such that
    $$
    \frac{M}{2}\rho < v(y,\tau) \leq c_1 \inf_{Q_{\delta \rho/2}(x_0,t_0)} v
    \leq c_1 v(x_0,t_0)= 2c_1 \vep \rho,
    $$
    which is a contradiction, taking $\vep < M/(4c_1)$.
\end{enumerate}

Therefore, there are $N>0$ and $r_0>0$, depending on $(x_0,t_0)$, such that
$$
 \sup_{Q_r(x_0,t_0)} u^- \geq N r, \quad \text{ for all } 0<r<r_0.
$$

\end{proof}

\section{Blow-up analysis}\label{s:blowup}

Recall that given a function $v: \R^n \times [-1,0]\to \R$, we define
    $$
    I(t; v) := \frac{1}{t}  \int_{-t}^0 \int_{\Rn} |\nabla v(x,s)|^2 G(x,-s)\, dx ds, \quad \text{ for } 0<t\leq 1,
    $$
where $G$ is the heat kernel. Moreover, we denote 
$$
v_r(x,t):= \frac{1}{r} v(rx,r^2 t), \quad \text{ for } 0< r < 1.
$$

The following lemma classifies all possible blow-up limits. 

\begin{lem}[Classification of blow-ups] \label{lem:blowups}
    Let $u$ be a Lipschitz solution of \eqref{eq:FBP} in $Q_1.$
Assume that $(0,0)\in \Gamma$. Consider the blow-up sequence $u_r$, for $0<r<1/2$. Then $u_r$ admits a subsequence converging uniformly in compact subsets of $\Rn\times(-\infty,0]$ and, for any convergent subsequence $u_{r_j}$, there exist a constant $a \geq 0$ and a unit vector $e$ such that 
$$
u_{r_j} \to a (x\cdot e), \quad \text{ as } j \to \infty.
$$
Moreover, the following dichotomy holds:
\begin{enumerate}[$(i)$]
\item If $(0,0)$ is a regular point, then $a>0$.
\item If $(0,0)$ is a singular point, then $a=0$. In particular, the blow-up limit is unique.
\end{enumerate}
\end{lem}

As an immediate consequence, we obtain spatial differentiability of $u$ at singular points. 
\begin{cor}
If $(0,0)$ is a singular point, then $u$ is differentiable in $x$ at $(0,0)$, and 
$$|\nabla u(0,0)|=0.$$
\end{cor}

\begin{proof}
By Lemma~\ref{lem:blowups}, it follows that
$$
\lim_{r\to0} \sup_{Q_1} |u_r| =0.
$$
Otherwise, there exist $\vep>0$ and a sequence $r_j\to0$ such that 
$$\sup_{Q_1} |u_{r_j}|\geq\varepsilon,$$ 
but, up to a subsequence,  $u_{r_j}\to0$ locally uniformly, which is a contradiction. Therefore,
$$
 \sup_{(x,t)\in Q_r} \frac{|u(x,t)-u(0,0)|}{r} \to 0, \qquad \text{ as } r\to 0.
$$
In particular, $u(x,0)=o(|x|)$ as $x\to0$, so $u$ is differentiable in $x$ at $(0,0)$, and 
$$|\nabla u(0,0)|=0.$$
\end{proof}

We will see in Section~\ref{sec:regularity} that $u\in C^{1,\alpha}$ at regular points, thus proving Theorem~\ref{thm:main2}.

\begin{proof}[Proof of Lemma~\ref{lem:blowups}]
    We divide the proof of the lemma into two main steps.\medskip

    \textbf{Step 1:} \textit{Energy of the blow-up limits.}  
    Let $D:=\Rn\times(-\infty,0]$. Since $u$ is Lipschitz in $Q_1$, the family $\{u_{r}\}$ is Lipschitz, uniformly in $r$, on compact subsets of $D$. Hence, by the Arzel\`a--Ascoli theorem, after passing to a subsequence, we have
$$
\begin{array}{cl}
u_{r_j}\to u_0 & \quad \text{locally uniformly on } D,\\[0.2cm]
\nabla u_{r_j}\rightharpoonup \nabla u_0 &\quad  \text{weakly in } L^2_{\rm loc}(D).
\end{array}
$$

Since the equations are invariant under the Lipschitz scaling, and viscosity solutions are stable under uniform convergence, $u_0$ is a global Lipschitz viscosity solution of the same problem as $u$. 
Also, $u_{r_j}(0,0)=0$ because $(0,0)\in \Gamma$, so $$u_0(0,0)=0.$$

Our goal is to show that the energy of $u_0$ is constant.
We prove this in two substeps:\medskip

    \underline{Step 1.1:} \textit{Convergence of the energy.} 
Let $w :=u\psi$,  where $0\leq \psi\leq 1$ is a smooth cut-off function in $x$, i.e., $\psi\equiv 0$ outside $B_{2/3}$ and $\psi\equiv1$ in  $B_{1/2}$.
Note that 
    $$
    w_{r_j}(x,t)=(u \psi)_{r_j}(x,t) = u_{r_j}(x,t) \psi(r_jx).
    $$
   We aim to prove that, for all $0<t\leq 1$,
    \begin{equation} \label{eq:convenergy}
    \lim_{j\to \infty} J(t; w_{r_j}^+, w_{r_j}^-) = J(t; u_0^+, u_0^-).
    \end{equation}
  Since $\psi$ is bounded in $C^1(\R^n)$, and $\psi(r_j x)\to \psi(0)=1$, as $j\to \infty$, by Step 1,
    $$
\begin{array}{cl}
w_{r_j}\to u_0 & \quad \text{locally uniformly in } D,\\[0.2cm]
\nabla w_{r_j}\rightharpoonup \nabla u_0 &\quad  \text{weakly in } L^2_{\rm loc}(D).
\end{array}
$$
If $\nabla w_{r_j}\to \nabla u_0$ a.e. in $D$, then using  $|\nabla w_{r_j}|\leq C$ and 
$$
\frac{1}{t} \int_{-t}^0 \int_{\Rn} G(x,-s)\, dxds=1,
$$
 by the dominated convergence theorem,
\begin{align*}
\lim_{j\to \infty} I(t; w_{r_j}^\pm) & =  \frac{1}{t} \lim_{j\to \infty}  \int_{-t}^0 \int_{\Rn}  |\nabla w_{r_j}^\pm|^2 G(x,-s)\, dxds \\
&= \frac{1}{t} \int_{-t}^0 \int_{\Rn}  |\nabla u_0^\pm|^2 G(x,-s)\, dxds\\
&= I(t; u_0^\pm).
\end{align*}
Therefore,
$$
  \lim_{j\to \infty} J(t; w_{r_j}^+, w_{r_j}^-) =   \lim_{j\to \infty}  I(t; w_{r_j}^+)  I(t; w_{r_j}^-) =  I(t; u_0^+) I(t; u_0^-) =J(t; u_0^+, u_0^-),
$$
thus proving \eqref{eq:convenergy}.\medskip

Now, let $Q$ be a bounded parabolic cylinder in the interior of $D$. We will show that 
    \begin{equation}\label{eq:strongL2}
        \lim_{j\to \infty} \|\nabla w_{r_j}\|_{L^2(Q)} = \|\nabla u_0\|_{L^2(Q)}.
    \end{equation}
Note that weak convergence and \eqref{eq:strongL2} imply that $\nabla w_{r_j}\to \nabla u_0$ strongly in $L^2(Q)$, and thus, up to a subsequence, almost everywhere. Hence, it remains to establish \eqref{eq:strongL2}.\medskip

Fix $h>0$. We claim that:
\begin{enumerate}[$(i)$]
    \item \textit{Small energy near the 0-level set:} For $j$ sufficiently large,
     $$\iint_{Q\cap \{ |w_{r_j}|< h\} }|\nabla w_{r_j}|^2\, dxdt\leq Ch,$$
     where $C$ is independent of $j$, and $h$. \medskip
     
    \item \textit{Convergence away from the 0-level set:} As $j\to \infty$,
     $$\iint_{Q\cap \{ |w_{r_j}|\geq h\} }|\nabla w_{r_j}|^2\, dxdt\to \iint_{Q\cap \{ |u_0| \geq h\} }|\nabla u_0|^2\, dxdt.$$
\end{enumerate}

We begin proving $(i)$. For $Q$ fixed, choose a larger cylinder $\widetilde Q$, in the interior of $D$, and containing $\overline Q$. For sufficiently large $j$, we have $w_{r_j}=u_{r_j}$ in $\widetilde Q$. 
Set $$v_j:=w_{r_j}^{+}.$$ 
By Remark~\ref{rem:ordering}, $v_j$ is subcaloric in $\widetilde Q$, in the distributional sense. Hence,
$$
\mu_j:=\Delta v_j-\partial_t v_j
$$
is a nonnegative Radon measure. Fix $0<\vep<h$, and define the functions
$$
\mathcal T_{\vep,h}(\tau) :=\min\{\max\{\tau,\vep\},h\}, \quad
\mathcal P_{\vep,h}(\tau) :=\int_0^\tau \mathcal T_{\vep,h}(s)\,ds, \quad \text{ for } \tau \geq 0.
$$
We have
$$
0\leq \mathcal T_{\vep,h}(\tau)\leq h, \quad 0\le \mathcal P_{\vep,h}(\tau)\leq h\tau,
$$
and
$$
\mathcal T_{\vep,h}'(\tau) =\chi_{\{\vep<\tau<h\}} \quad\text{for a.e. }\tau.
$$

We recall the following chain rule for $C^2$-functions, which holds almost everywhere:
\begin{equation} \label{eq:chainrule}
(\Delta - \partial_t) \mathcal P_{\vep,h}(v) = \mathcal T_{\vep, h}' (v) |\nabla v|^2 + \mathcal T_{\vep, h} (v) (\Delta - \partial_t )v.
\end{equation}
Applying \eqref{eq:chainrule} to space-time mollifications of $v_j$, and passing to the limit, we see that
\begin{equation} \label{eq:productrule}
\operatorname{div} \left( \mathcal T_{\vep,h}(v_j)\nabla v_j \right) -\partial_t \mathcal P_{\vep,h}(v_j) = \chi_{\{\vep<v_j<h\}}|\nabla v_j|^2 +\mathcal T_{\vep,h}(v_j)\mu_j,
\end{equation}
in the distributional sense on $\widetilde Q$.
Choose $\varphi\in C_c^\infty(\widetilde Q)$ such that $\varphi\geq 0$ and $\varphi=1$ on a neighborhood of $\overline Q$. Testing \eqref{eq:productrule} against $\varphi$, and using that $\mathcal T_{\vep,h}(v_j)\mu_j$ is nonnegative, we get
\begin{align*}
\iint_{Q\cap \{\vep<v_j<h\}} |\nabla v_j|^2\, dxdt &\leq -\iint_{\widetilde Q} \mathcal T_{\vep,h}(v_j)\nabla v_j\cdot\nabla\varphi +\iint_{\widetilde Q}\, 
\mathcal P_{\vep,h}(v_j)\partial_t \varphi\, dxdt\\
&\leq  h \|\nabla v_j\|_{L^\infty(\widetilde Q)} \iint_{\widetilde Q}|\nabla\varphi|\,dxdt +h \|v_j\|_{L^\infty(\widetilde Q)} \iint_{\widetilde Q}|\partial_t \varphi|\, dxdt\\
&\leq C  h,
\end{align*}
where $C>0$ is independent of $j$, $\vep$, and $h$.

Letting $\vep\to 0$, and using that $\nabla v_j=0$~a.e.~on level sets, we obtain
$$
\iint_{Q\cap\{0\le v_j< h\}}|\nabla v_j|^2\le C h.
$$
The same argument applies to \(w_{r_j}^{-}\). Adding the two estimates proves $(i)$.\medskip

Next, we prove $(ii)$.  Since $w_{r_j}=u_{r_j}$ in $\widetilde Q$, $w_{r_j}$ is a Lipschitz viscosity solution of 
\begin{equation*}
    \begin{cases}
        \partial_t w_{r_j}-F^-(D^2w_{r_j})=0 &\text{ in } \{w_{r_j}>0\}\cap \widetilde Q,\\
        \partial_t w_{r_j}-F^+(D^2w_{r_j})=0 &\text{ in }  \{w_{r_j}<0\}\cap \widetilde Q.
    \end{cases}
\end{equation*}
By the interior $C^{1,\alpha}$ regularity for uniformly parabolic equations, and the fact that $w_{r_j}\to u_0$ locally uniformly in $D$, we have
$$
\nabla w_{r_j}\to \nabla u_0 \quad \text{ locally uniformly in }  \{u_0\neq 0\} \cap D.
$$
Since $\nabla u_0=0$ a.e. in $\{ |u_0| =h\}$, the above convergence implies that
$$\chi_{\{|w_{r_j}|\geq h\}}|\nabla w_{r_j}|^2 \to \chi_{\{|u_0|\geq h\}}|\nabla u_0|^2\quad\text{a.e. in } Q.$$
It follows that
$$
    \lim_{j\to \infty}\iint_{Q\cap \{|w_{r_j}|\geq h\} }|\nabla w_{r_j}|^2\, dxdt=\iint_{Q\cap \{|u_0|\geq h\} }|\nabla u_0|^2\, dxdt,
$$
thus proving $(ii)$.\medskip

    Finally, by weak lower semicontinuity, and $(i)$ and $(ii)$, we conclude that
    \begin{align*}
 \iint_Q |\nabla u_0|^2\, dxdt &\leq \liminf_{j\to\infty} \iint_Q |\nabla w_{r_j}|^2\, dxdt\\
 & \leq  \limsup_{j\to\infty} \iint_Q |\nabla w_{r_j}|^2\, dxdt \\
 & \leq Ch  + \iint_{Q\cap \{ |u_0| \geq h\} }|\nabla u_0|^2\, dxdt\\
 & \leq Ch +  \iint_Q |\nabla u_0|^2\, dxdt.
 \end{align*}
Letting $h\to 0$, we obtain \eqref{eq:strongL2}.\medskip

\medskip
    \underline{Step 1.2:} \textit{Constant energy.} By the scaling property of $J$ (see Remark~\ref{rem:scaling}), we have
     $$
     J(t; w_{r_j}^+, w_{r_j}^-) = J(r_j^2 t; w^+, w^-).
     $$
     Moreover, by the local ACF formula (Theorem~\ref{thm:ACFlocal}), we know that the limit
     $$
    J(0^+;  w^+, w^-) := \lim_{t \to 0} J(t;  w^+, w^-) 
     $$
     is well-defined.
     Hence, by Step 1.1, for all $0<t\leq 1$,
     \begin{align*}
         J(t; u_0^+, u_0^-) &= \lim_{j\to \infty} J(t; w_{r_j}^+, w_{r_j}^-)
         = \lim_{j\to \infty} J(r_j^2 t; w^+, w^-)
         = J(0^+;  w^+, w^-),
     \end{align*}
     which implies that $J(t; u_0^+, u_0^-)$ is constant. 
\medskip

\textbf{Step 2:} \textit{The dichotomy.} We distinguish two cases:\medskip

\underline{Case 1:} Assume that $(0,0)$ is a regular point.  We claim that 
$$
        J(t; u_0^+, u_0^-) =: J_0>0.
$$
Indeed, by Step 1.2, we only need to show that the constant $J_0$ is strictly positive. Since $(0,0)$ is a regular point, by Lemma~\ref{lem:nondegeneracy}, there are constants $N>0$ and $r_0>0$ such that
$$
\sup_{Q_r} u^\pm \geq Nr, \quad \text{ for all } 0<r<r_0.
$$
Fix $R>0$. For $j$ sufficiently large so that $r_jR < r_0$, rescaling we have
$$
\sup_{Q_R} u_{r_j}^\pm = \frac{1}{r_j}\sup_{Q_{Rr_j}} u^\pm \geq NR.
$$
Passing to the limit as $j\to\infty$, we obtain
\begin{equation} \label{eq:nondeglimit}
\sup_{Q_R} u_0^\pm \geq NR,
\quad\text{for all }R>0.
\end{equation}
In particular, $u_0^\pm\not\equiv 0$ at every scale. We claim that, for every $0<t\leq 1$,
$$
I(t;u_0^\pm)>0.
$$ 
If for some $t_0>0$ we had $I(t_0;u_0^+)=0$, then $\nabla u_0^+=0$ a.e.~in $\mathbb R^n\times(-t_0,0]$, thus $u_0^+$ is constant on each time slice.
Hence, there is a continuous function $c\geq 0$ such that
$$
u_0^+(x,t)=c(t)\qquad \text{in } \R^n\times(-t_0,0].
$$
Then $D^2 u_0=0$ in the open set $\{u_0>0\}\cap (\mathbb R^n\times(-t_0,0])$.
Since $u_0$ satisfies
$$
\partial_t u_0 - F^-(D^2u_0)=0 \quad\text{in }\{u_0>0\},
$$
and $F^-(0)=0$, it follows that $c'(t)=\partial_t u_0=0$ in $\{u_0>0\}$, so $c$ is constant on each connected
component of $\{c>0\}=\{u_0>0\}$. 
Because $u_0(0,0)=0$, we have $c(0)=0$. If $\{c>0\}\neq \emptyset$, then $c$ would be a positive constant on one of its connected components. By continuity, $c$ vanishes at the right endpoint of this component, which is a contradiction. Therefore, $c\equiv0$ in $(-t_0,0]$.
This contradicts the non-degeneracy estimate in \eqref{eq:nondeglimit}, since
$$
0=\sup_{t\in(-R^2,0]} c(t)\geq NR,
$$
for $R$ sufficiently small. Hence, $I(t;u_0^+)>0.$ The same argument applies to $u_0^-$. Therefore,
\begin{equation*}
J(t;u_0^+,u_0^-)=I(t;u_0^+)I(t;u_0^-)=J_0>0.
\end{equation*}
By Proposition~\ref{prop:rigidity}, there are constants $a, b > 0$ and a unit vector $e$ such that 
\begin{equation}\label{e:two plane soln}
    u_{r_j} \to a (x\cdot e)^+ - b (x \cdot e)^-, \quad \text{ as } j \to \infty.
\end{equation}
We now show that $a = b$, in particular, $u_0:=\lim_{j\to \infty}u_{r_j}$ is of the following form:
$$
u_0(x)=a (x\cdot e).
$$
Suppose for contradiction that $a >b$ and suppose, for simplicity, that $e=e_n$. Let $\gamma:=\frac{a+b}{2}$ and $\delta:=\frac{a-b}{2}$. Then from \eqref{e:two plane soln},
\begin{equation*}
    u_0 (x) = \gamma x_n+\delta |x_n|.
\end{equation*}
Define $v(x,t):=\delta |x_n|$. Thus, $v=u_0-\gamma x_n$. Notice that $\partial_tv=\partial_t u_0$, $D^2v=D^2u_0$ distributionally, and hence $v$ satisfies the same global inequalities as $u_0$, i.e. 
\begin{equation*} 
    F^-(D^2v)\leq \partial_tv \leq F^+(D^2v),
\end{equation*}
in the viscosity sense.
But this is a contradiction as $v$ cannot be a viscosity supersolution, in particular, the first inequality cannot hold. Indeed, 
consider the test function $\varphi(x,t)=\delta x_n^2$, for $\delta$ as above. We have $\varphi(0,0)=0=v(0,0)$, and
$$
\varphi(x,t)\leq v(x,t) \qquad \text{for all } (x,t)\in Q_1, 
$$
i.e., $\varphi$ touches $v$ from below at $(0,0)$, but
\begin{equation*}
   0< 2\lambda \delta \leq F^-(D^2\varphi) \leq \partial_t\varphi =0, 
\end{equation*}
which is a contradiction. Thus, $a\leq b$. The argument holds analogously for $b>a$. Therefore, 
$$
u_{r_j}\to a(x\cdot e), \qquad \text{as } j\to \infty. \medskip
$$

\underline{Case 2:} Assume that $(0,0)$ is a singular point. Hence, there exist $r_j < 1/j$ such that $\sup_{Q_{r_j}}|u|< r_j/j^2$. Then, choosing $\rho_j := r_j / j$, for any $R > 0$ and all $j \geq R$,
$$
\sup_{Q_R} |u_{\rho_j}| =\frac{1}{\rho_j}  \sup_{Q_{R\rho_j}} |u| \leq \frac{1}{\rho_j}  \sup_{Q_{j\rho_j}} |u| < \frac{1}{j}\to 0, \quad \text{ as } j\to\infty.
$$
By Step 1, after passing to a subsequence, we have
$$
u_{\rho_j}\to v_0\equiv 0, \quad \text{ as } j\to\infty,
$$
locally uniformly in $\Rn\times(-\infty,0].$ Therefore,
$$
J(t;v_0^+,v_0^-):=J_0=0.
$$

Now, let $u_0$ be any blow-up limit. We will show that $u_0\equiv 0$. Fix $T>0$. Applying the argument of Step 1.2 on the strip $\Rn\times(-T,0]$, we obtain
$$
J(T;u_0^+,u_0^-)=J_0=0.
$$
Hence, $I(T;u_0^+)=0$ or  $I(T;u_0^-)=0.$ Assume that $I(T;u_0^+)=0$. The other case is analogous.
Then, $\nabla u_0^+=0$ a.e. in $\Rn\times(-T,0]$, so there is a continuous function $c\geq 0$ such that
$$
u_0^+(x,t)=c(t)
\qquad
\text{in }\Rn\times(-T,0].
$$
Arguing as in Case 1, we see that $c\equiv 0$, and thus, $u_0\leq 0$ in $\Rn\times(-T,0]$. Moreover, by the global inequalities,
$$
\partial_t u_0-F^+(D^2u_0)\leq 0
\qquad
\text{in }\Rn\times(-T,0].
$$
Since $u_0(0,0)=0$, by the parabolic strong maximum principle,
$$
u_0\equiv 0
\qquad
\text{in }\Rn\times(-T,0].
$$
Because $T>0$ was arbitrary, $u_0\equiv 0$ in $\Rn\times(-\infty,0]$. This holds for every blow-up limit, so
the blow-up limit is unique and $a=0$, as claimed.
\end{proof} 

\begin{rem}
The ordering assumption plays a key role in the proof: it allows us to apply the parabolic monotonicity formula to classify all blow-up limits. Without this assumption, the classification of global Lipschitz solutions to two-phase anisotropic problems is an open problem. In the elliptic setting, this was established by Caffarelli, De Silva, and Savin \cite{CDSS18}, in dimension 2, for problems involving different linear operators; see also \cite{DSS19}.
\end{rem}

\section{Improvement of flatness}\label{sec:IF}

In the previous section, we showed that if $(x_0,t_0)$ is a regular point, then 
$$
u_{r_j}\to a(x\cdot e), \quad a>0, \quad  \text{as } j\to \infty.
$$
Now, a priori, the slope $a$ and direction $e$ could depend on the sequence of scales $r_j$. We will prove that the blow-up limit is unique by showing an improvement of flatness of the graph of $u$ at every small scale, in the  spirit of De Silva's work in \cite{DS11}.

Recall that, given $(x_0,t_0)\in \Gamma$, we define the parabolic Lipschitz rescaling as
$$
u_\rho(x,t):=\frac{1}{\rho}u(\rho x+x_0,\rho^2 t+t_0).
$$
Note that if $(x_0,t_0)$ is a regular point for $u$, then $(0,0)$ is a regular point for $u_\rho$.
The following lemma provides the initial flatness at a sufficiently small scale.

\begin{lem} \label{lem:closeness}
    Let $(x_0,t_0)\in \Gamma$ be a regular point.
    For all $\delta>0$, there is $\rho\in (0,1)$ such that
    \begin{equation*} 
            \sup_{Q_1} \left| u_\rho -a(x\cdot e)\right|\leq \delta,
    \end{equation*}
    for some constant $a>0$ and unit vector $e$.
\end{lem}

\begin{proof}
Let $(x_0,t_0)\in \Gamma$ be a regular point.  
    By Lemma~\ref{lem:blowups}, there is a subsequence of blow-ups $u_{r_j}$ such that
    $u_{r_j}\to a(x\cdot e)$ locally uniformly as $j\to \infty$, for some constant $a>0$ and unit vector $e$.
    Let $\delta>0$. Then there is $j_0=j_0(\delta)\geq 1$ such that
    $$
    \sup_{Q_1}| u_{r_j}- a(x\cdot e)|\leq  \delta, \qquad \text{ for all } j\geq j_0.
    $$
 Setting $\rho:=r_{j_0}$, we obtain the result.
\end{proof}

\begin{rem}\label{rem:normal}
By Lemma~\ref{lem:nondegeneracy}, we have $a\geq N$, where $N$ is the non-degeneracy constant. Hence, dividing the above flatness condition by $a$,
  $$
    \sup_{Q_1}| a^{-1} u_\rho- x\cdot e|\leq  a^{-1} \delta \leq N^{-1} \delta.
    $$
Since $a^{-1} u_\rho$ satisfies the same equations and inequalities as $u$, applying Lemma~\ref{lem:closeness} with $N\delta$ in place of $\delta$, and denoting the normalized function $a^{-1}u_\rho$ again by $u$, we may assume that $(0,0)$ is a regular point of $\Gamma$ and 
$$
\sup_{Q_1} |u - x\cdot e| \leq \delta.
$$
\end{rem}

The key step is showing that the flatness improves geometrically at a smaller scale. Namely:

\begin{thm}[Improvement of flatness] \label{thm:IF}
     Let $u$ be a Lipschitz solution of \eqref{eq:FBP} in $Q_1$.
Let $e\in \Rn$ be a unit vector. Assume that $(0,0)\in \Gamma$ and that $u$ satisfies
    \begin{equation} \label{eq:flatness}
            \sup_{Q_1} |u - x\cdot e| \leq \delta,
    \end{equation}
    for some $\delta>0$. There are universal constants $0<r_0\leq 1/2$ and $C_0>0$ such that for all $0<r\leq r_0$, there is $\delta_0=\delta_0(r)>0$, such that: if $0 < \delta \leq \delta_0$, then there is some vector $ \nu \in \Rn$, with  $| \nu-e|\leq C_0\delta$, such that $u$ satisfies
    $$
    \sup_{Q_r} |u- x \cdot \nu| \leq \frac{\delta}{2}r.
    $$
\end{thm}

\begin{proof}
After a rotation, we may assume that $e=e_n$.
    We divide the proof into three steps:\medskip

    \textbf{Step 1:} \textit{Compactness.} Assume for the sake of contradiction that the statement is false. Fix $0<r_0\leq 1/2$ and $C_0>0$ to be chosen later. 
    There exists some $0<r\leq r_0$, a sequence $\delta_k\to 0$, and viscosity solutions $u_k$ to \eqref{eq:FBP} in $Q_1$ such that $(0,0)\in \Gamma_k$, where $\Gamma_k$ is the free boundary corresponding to the function $u_k$, and
    $$
    \sup_{Q_1} |u_k -  x_n| \leq \delta_k,
    $$
    but the conclusion of the theorem does not hold, i.e., 
     \begin{equation} \label{eq:contradiction}
    \sup_{Q_r} |u_k- x\cdot \nu| > \frac{\delta_k}{2}r,
    \end{equation}
    for all vectors $\nu\in \R^n$ with $|\nu- e_n|\leq C_0 \delta_k$. Let
    $$
    v_k:= \frac{u_k -  x_n}{\delta_k}.
    $$
    By the flatness condition, we have $|v_k|\leq 1$ in $Q_1$, and $\Gamma_k \subseteq Q_1\cap \{|x_n|\leq \delta_k\}$.
    Moreover, 
    $$
    F^-(D^2 v_k)   \leq  \partial_t v_k \leq F^+(D^2 v_k) \quad \text{ in } Q_1, 
    $$
    in the viscosity sense. By the interior regularity for functions in the Pucci class,
    $$
    \|v_k\|_{C^{0,\gamma}(\overline Q_{7/8})} \leq C, 
    $$
    where $C>0$ does not depend on $k$. By the Arzel\`a-Ascoli theorem, up to a subsequence,
    \begin{equation} \label{eq:unifconv}
            v_k \to w \quad \text{ uniformly on } \overline Q_{7/8}.
    \end{equation}
    Note that $w(0,0)=0$ and $|w|\leq 1$ on $\overline Q_{7/8}$.
\medskip

    \textbf{Step 2:} \textit{Limiting problem.} 
    Next, we show that $w$ is a viscosity solution to
    $$
        \begin{cases}
\partial_t w- F^-(D^2 w)  = 0 & \text{ in } Q_{3/4}^+,\\
\partial_t w-F^+(D^2 w)= 0 & \text{ in } Q_{3/4}^-,\\
\partial_t w- F^-(D^2 w)  \geq 0 & \text{ in } Q_{3/4},\\
 \partial_t w - F^+(D^2 w)   \leq 0 & \text{ in } Q_{3/4},
 \end{cases}
 $$
     where $Q_{3/4}^{\pm}=Q_{3/4}\cap\{\pm x_n>0\}$.
It is enough to see that $w$ is a viscosity supersolution of the first equation since the argument is analogous for the other cases.

Indeed, assume that $\varphi\in C^2$ touches $w$ from below at $(x_0,t_0)\in Q_{3/4}^+$.
 We may assume that they touch strictly; otherwise, we take $\varphi -\vep |x-x_0|^2-\vep(t_0-t)$, for some $\vep>0$ small.
We need to prove that
\begin{equation} \label{eq:goal}
\partial_t \varphi(x_0,t_0)- F^-(D^2 \varphi(x_0,t_0))  \geq 0.
\end{equation}
By the uniform convergence in \eqref{eq:unifconv}, there are points $(x_k,t_k)\in \{u_k>0\}$ converging to $(x_0,t_0)$, and a sequence of real numbers
$d_k\to 0$, such that the function
$\varphi_k := \varphi + d_k$ touches $v_k$ from below at $(x_k,t_k)$. Since 
$\partial_t v_k- F^-(D^2 v_k)  = 0$ in $\{u_k>0\}$, in the viscosity sense,
$$
\partial_t \varphi_k(x_k,t_k)- F^-(D^2  \varphi_k(x_k,t_k))  \geq 0.
$$
Passing to the limit as $k\to \infty$, and using continuity, we get \eqref{eq:goal}.

\medskip
    \textbf{Step 3:} \textit{Contradiction.} By Lemmas~\ref{lem:Pucci_implies_transmission} and \ref{lem:blowup_C1_alpha}, there is $C_0>0$ such that 
    $$
    \sup_{Q_r} |w - \nabla w(0,0) \cdot x| \leq C_0 r^{1+\alpha},
    $$
    where $|\nabla w(0,0)|\leq C_0.$ Using \eqref{eq:unifconv}, it follows that
     $$
    \sup_{Q_r} |v_k - \nabla w(0,0)\cdot x| \leq 2C_0 r^{1+\alpha},
    $$
    for $k$ sufficiently large. Rescaling back, we get
    $$
        \sup_{Q_r} |u_k -  x_n-\delta_k \nabla w(0,0) \cdot x| \leq 2C_0 r^{1+\alpha}\delta_k.
    $$
    Let $\nu_k= e_n+\delta_k\nabla w(0,0)$. Then $|\nu_k-e_n|\leq C_0 \delta_k$, and choosing $0<r_0\leq 1/2$ small enough so that $2C_0 r^\alpha\leq 1/2$, we see that
    $$
    \sup_{Q_r} |u_k- x\cdot \nu_k| \leq \frac{\delta_k}{2}r,
    $$
    which contradicts \eqref{eq:contradiction}.
\end{proof}

\section{$C^{1,\alpha}$ regularity} \label{sec:regularity}

An iteration of the improvement of flatness yields the $C^{1,\alpha}$ regularity of the function $u$ at \textit{flat} free boundary points, i.e., points in $\Gamma$ where the flatness condition \eqref{eq:flatness} is satisfied, so in particular, at regular points.
The proof is standard, but we provide it for completeness.

\begin{thm} \label{thm:regularity}
 Let $0<r_0 \leq 1/2$ and $C_0>0$ be as in Theorem~\ref{thm:IF}. There exist $\bar \delta_0=\bar\delta_0(r_0,C_0)>0$ and $\alpha = \alpha(r_0) \in (0,1)$ such that
if $u$ is a viscosity solution to \eqref{eq:FBP} in $Q_1$ satisfying \eqref{eq:flatness} for some $0<\delta \leq \bar \delta_0$, then for any $(x_0,t_0)\in \Gamma\cap Q_{1/2}$, there is $\nu_{x_0,t_0} \in \Rn$, with $|\nu_{x_0,t_0}- e|\leq C\delta$, such that
$$
\sup_{Q_r(x_0,t_0)} |u - \nu_{x_0,t_0} \cdot (x-x_0)| \leq C \delta r^{1+\alpha}, \quad \hbox{ for all}~ r\leq  r_0,
$$ 
where $C>0$ depends on $r_0$ and $C_0$.
\end{thm}

\begin{proof}
Let $(x_0,t_0)\in \Gamma\cap Q_{1/2}$.
For $j\geq 0$, set $\rho_j := \frac{1}{2}r_0^j$, and consider the blow-up sequence
$$
u_{\rho_j}(x,t) = \frac{1}{\rho_j} u( \rho_j x+x_0, \rho_j^{2} t+t_0)\quad \hbox{ for}~ (x,t)\in Q_1.
$$
Then $u_{\rho_j}$ is a viscosity solution of \eqref{eq:FBP}, with $u_{\rho_j}(0,0)=0$. 
Define  $\alpha\in (0,1)$ such that $r_0^{\alpha} = 1/2.$ 
Let $\hat \delta_0= \hat \delta_0(r_0)$ be the flatness threshold in Theorem~\ref{thm:IF}. Assume that
    \begin{equation} \label{eq:flatness2}
            \sup_{Q_1} |u - x\cdot e| \leq \delta,
    \end{equation}
    for some unit vector $e\in \Rn$ and some
$$
0<\delta \leq \bar\delta_0 := \min \left\{ \frac{\hat \delta_0}{8}, \frac{1}{16 C_0}\right\}.
$$
Set $\delta_j:=2^{2+\alpha}\delta\rho_j^\alpha.$ 
We claim that for all $j\geq 0$, there exists $\nu_j\in\Rn$ such that
$$
\frac{1}{2}\leq|\nu_j|\leq\frac{3}{2}, \qquad |\nu_{j+1}-\nu_j|\leq C_0\delta_j,
$$
and
\begin{equation} \label{eq:flatnessj}
\sup_{Q_1}|u_{\rho_j}-x\cdot\nu_j|\leq\delta_j.
\end{equation}
We proceed by induction. For $j=0$, take $\nu_0=e$. We have $|\nu_0|=1$. Also, since $u(x_0,t_0)=0$, by the flatness condition \eqref{eq:flatness2}, applied to $x=x_0$, we get $|x_0\cdot e|\leq\delta$. Hence,
\begin{align*}
\sup_{Q_1}|u_{\rho_0}-x\cdot e| &\leq 2\sup_{Q_{1/2}(x_0,t_0)}|u-x\cdot e| +2|x_0\cdot e| \leq 4\delta = \delta_0,
\end{align*}
where we used $Q_{1/2}(x_0,t_0)\subset Q_1$. Assume that the claim holds up to some $j\geq0$.
Define
$$
v_j:=\frac{u_{\rho_j}}{|\nu_j|}, \qquad e_j:=\frac{\nu_j}{|\nu_j|}.
$$
Then by \eqref{eq:flatnessj}, and $|\nu_j|\geq 1/2$,
$$
\sup_{Q_1} |v_j-x\cdot e_j| \leq \frac{\delta_j}{|\nu_j|}  \leq 2\delta_j = 8 r_0^{j \alpha} \delta \leq 8\delta\leq \hat \delta_0.
$$
By Theorem~\ref{thm:IF}, there exists $\nu\in\Rn$ such that
$$
|\nu-e_j|\leq C_0\frac{\delta_j}{|\nu_j|}, \qquad
\sup_{Q_{r_0}}|v_j-x\cdot\nu| \leq\frac{\delta_j}{2|\nu_j|}r_0.
$$
Setting $\nu_{j+1}:=|\nu_j|\nu$, we obtain
$$
|\nu_{j+1}-\nu_j| =|\nu_j|\,|\nu-e_j| \leq C_0\delta_j.
$$
Noting that $\delta_j=4 \delta r_0^{\alpha j}= 4\delta 2^{-j}$,
$$
|\nu_{j+1}-e| \leq C_0\sum_{k=0}^j\delta_k \leq 8C_0\delta \leq\frac{1}{2}.
$$
Since $|e|=1$, we have
$
\frac{1}{2} \leq |\nu_{j+1}| \leq \frac{3}{2}.
$
Moreover, $\rho_{j+1}=r_0\rho_j$, thus
\begin{align*}
\sup_{Q_1}|u_{\rho_{j+1}}-x\cdot\nu_{j+1}| &=\frac{1}{r_0}\sup_{Q_{r_0}}|u_{\rho_j}-x\cdot\nu_{j+1}| 
=\frac{|\nu_j|}{r_0}\sup_{Q_{r_0}}|v_j-x\cdot \nu| 
\leq\frac{\delta_j}{2} =\delta_{j+1},
\end{align*}
which proves the claim for $j+1$. By induction, the claim holds for all $j\geq 0$.\medskip

In particular, there exists some $\nu_{x_0,t_0}\in \Rn$ such that $\nu_j \to \nu_{x_0,t_0}$ as $j\to \infty$, and 
\begin{align*}
|\nu_{x_0,t_0} -  e| \leq \sum_{j=0}^\infty |\nu_{j+1}-\nu_j|\leq C_0  \sum_{j=0}^\infty \delta_j\leq 8 C_0 \delta.
\end{align*}
Given $0<r \leq r_0$, there exists some $j\geq 0$ such that $\rho_{j+1} < r \leq \rho_j$. We conclude that:
\begin{align*}
\sup_{Q_r(x_0,t_0)} |u - \nu_{x_0,t_0} \cdot (x-x_0)| & 
\leq \rho_j \sup_{Q_1} |u_{\rho_j} - \nu_{x_0,t_0}  \cdot x| \\
& \leq \rho_j \sup_{Q_1} |u_{\rho_j} -  x \cdot \nu_j| + \rho_j |\nu_{x_0,t_0}- \nu_j|\\
& \leq \rho_j \delta_j  +   \rho_j C_0 \sum_{k=j}^\infty \delta_k \\
&= 2^{2+\alpha} \delta \rho_j^{1+\alpha} + 2^{3+\alpha} C_0 \delta \rho_j^{1+\alpha}\\
&  \leq C\delta r^{1+\alpha},
\end{align*}
where $C>0$ depends only on $r_0$ and $C_0$.
\end{proof}

\begin{rem}
    The estimate in Theorem~\ref{thm:regularity} implies that $u$ has a parabolic $C^{1,\alpha}$ expansion at every $(x_0,t_0)\in \Gamma\cap Q_{1/2}$, where $\nu_{x_0,t_0}=\nabla u(x_0,t_0)\neq 0$. Moreover, a simple computation shows that the gradient of $u$ is continuous on $\Gamma\cap Q_{1/2}$ (e.g., see \cite[Corollary~6.3]{KSC26}). 
\end{rem}

We state this in the next result.

\begin{cor}\label{cor:normal_Holder}
    Assume the hypotheses of Theorem~\ref{thm:regularity}.
   For any $(x_1,t_1), (x_2,t_2) \in \Gamma\cap Q_{1/2}$, with $r:=d_p((x_1,t_1), (x_2,t_2))\leq r_0/3$, it holds that
    $$
    |\nu_{x_1,t_1}-\nu_{x_2,t_2}|\leq C\delta {r^{\alpha}.}
    $$
\end{cor}

It remains to show that the set of regular points in $\Gamma$, which we call $\mathcal{R}$, is relatively open and locally a $C^{1,\alpha}$ hypersurface.
These results will be a consequence of Theorem~\ref{thm:regularity}.

\begin{thm} 
     The set $\mathcal{R}$ is relatively open in $\Gamma$. Moreover, it is locally a $C^{1,\alpha}$ hypersurface.
\end{thm}

\begin{proof}
     Fix an arbitrary point in $\mathcal{R}$. After a translation, we may assume that the point is the origin. We will see that $\Gamma\cap Q_{\rho}\subseteq \mathcal{R}$ for a small $\rho > 0$, which implies that $\mathcal{R}$ is relatively open in $\Gamma$. 
     Indeed, since $(0,0)$ is a regular point, after a rescaling, $u$ satisfies the flatness condition \eqref{eq:flatness} with a small $\delta > 0$ to be chosen later. Hence, by Theorem~\ref{thm:regularity}, for any $(x_0,t_0)\in \Gamma\cap Q_{1/2}$, there is $\nu_{x_0,t_0} \in \Rn$  such that
$$
\sup_{Q_r(x_0,t_0)} |u- \nu_{x_0,t_0} \cdot (x-x_0)| \leq C \delta r^{1+\alpha}, \quad \hbox{ for all}~ r\leq  r_0.
$$ 
Take $y:=x_0+\frac{r}{2} \cdot\frac{\nu_{x_0,t_0}}{|\nu_{x_0,t_0}|}\in B_r(x_0)$. Then $(y,t_0)\in Q_r(x_0,t_0)$, and the above estimate implies
$$
|u(y,t_0)| \geq \frac{r}{2} |\nu_{x_0,t_0}| - C  r^{1+\alpha} \geq \frac{r}{4}- C  r^{1+\alpha}.
$$
Taking $\bar r\leq r_0$ small enough so that $C \bar r^{\alpha}\leq 1/8$, it follows that 
$$
\sup_{Q_{r}(x_0,t_0)} |u| \geq \frac{1}{8} r, \quad \text{ for all } r\leq \bar r.
$$
Hence, $(x_0,t_0)\in \mathcal{R}$, which implies that $\mathcal{R}$ is relatively open.\medskip

We now prove that $\mathcal R $ is locally a $C^{1,\alpha}$ hypersurface. 
Assume without loss of generality that $\nu_{0,0} = e_n$, rescale to replace $\bar r$ by $1$, and choose $\delta$ small enough so that, for all $(x_0,t_0) \in \Gamma\cap Q_1$, 
\begin{equation} \label{eq:closenormals}
    |\nu_{x_0,t_0}-e_n| \leq \frac{1}{3},
\end{equation}
and
\begin{equation} \label{eq:parexpansion}
\sup_{Q_r(x_0,t_0)} |u - \nu_{x_0,t_0} \cdot (x-x_0)| \leq \frac{1}{6} r^{1+\alpha}, \quad \hbox{ for all}~ r\leq 1.
\end{equation}
In particular, setting $r = 1$ and $(x_0,t_0)$ at the origin, we obtain that on $\{x_n = \frac{1}{2}\}\cap Q_1$, $|u - \frac{1}{2}| \leq \frac{1}{6}$, and thus, $u > 0$. Analogously, $u < 0$ on $\{x_n = -\frac{1}{2}\}\cap Q_1$. Hence, by continuity, for each $x' \in B'_{1/2}$ and $t \in (-\frac{1}{4},0)$, there exists $x_n^*(x',t) \in (-\frac{1}{2},\frac{1}{2})$ such that $$u(x',x_n^*(x',t),t) = 0.$$

We claim that this zero is unique and that $x_n^*(x',t)$ is Lipschitz (in the parabolic sense). Indeed, by \eqref{eq:closenormals},
for any  unit vector $\sigma \in \R^n$ such that $e_n\cdot\sigma \geq \frac{2}{3}$,
$$\nu_{x_0,t_0}\cdot\sigma = e_n\cdot \sigma +(\nu_{x_0,t_0}-e_n)\cdot \sigma \geq \frac{2}{3} - \frac{1}{3} \geq \frac{1}{3}.$$
Also, for any $s \in (0,1)$ such that $x_0+s\sigma \in B_1$, and any $\tau \in [0,\frac{1}{2})$ such that $t_0-s^2\tau > -1$, applying \eqref{eq:parexpansion} at the point $(x_0+s\sigma,t_0-s^2\tau)\in Q_s(x_0,t_0)$, 
$$u\left(x_0+s\sigma,t_0-s^2\tau\right) \geq s (\nu_{x_0,t_0}\cdot \sigma) - \frac{1}{6}s^{1+\alpha} \geq \frac{1}{3}s-\frac{1}{6}s^{1+\alpha} > 0.$$
Analogously, if $x_0-s\sigma \in B_1$, then 
$$u\left(x_0-s\sigma,t_0-s^2\tau\right) < 0.$$ 
Letting $s$, $\sigma$ and $\tau$ vary, we deduce that
\begin{equation}\label{eq:conecondition}
\Gamma\cap Q_1(x_0,t_0) \subset \left(x_0,t_0\right)+\mathcal{C},
\end{equation}
where
$$\mathcal{C} := \left\{(x,t) \in \R^n\times\R : \, t \leq 0, \ |x_n| \leq \tfrac{2}{\sqrt{5}}|x'| + \sqrt{2}|t|^{1/2}\right\}.$$

Let $D_{1/2} := B_{1/2}'\times(-1,1) \times (-\frac{1}{4},0)$. Combining the existence and cone condition \eqref{eq:conecondition}, we obtain that $\Gamma\cap D_{1/2}$ is the graph of a parabolic Lipschitz function, i.e.,
$$\Gamma\cap D_{1/2} = \left\{(x',x_n^*(x',t),t) : \, x' \in B'_{1/2}, t\in \left(-\tfrac{1}{4},0\right)\right\},$$
with 
$$|x_n^*(x_1',t_1) - x_n^*(x_2',t_2)| \leq \frac{2}{\sqrt{5}}|x_1'-x_2'| + \sqrt{2}|t_1-t_2|^{1/2}.$$

Finally, we prove that this graph is $C^{1,\alpha}$ in the parabolic sense. For any $(x_1,t_1), (x_2,t_2)$ in $\Gamma \cap D_{1/2}$, with $t_1 \leq t_2$, by the $C^{1,\alpha}$ estimate of $u$ at $(x_2,t_2)$, and since $u(x_1,t_1)=0$,
$$
|\nu_{x_2,t_2} \cdot (x_1-x_2)| \leq C\big(|x_1-x_2|+ |t_1-t_2|^{1/2}\big)^{1+\alpha}.
$$
Write
$
\nu_{x_2,t_2} :=n_{x_2,t_2}e_n+\vep_{x_2,t_2},
$
where $n_{x_2,t_2}\geq 2/3$ and $\vep_{x_2,t_2}\cdot e_n=0$. Since $x_{i,n}=x_n^*(x_i',t_i)$,
$$
\Big|n_{x_2,t_2}\big(x_n^*(x_1',t_1)-x_n^*(x_2',t_2)\big)+\varepsilon_{x_2,t_2}\cdot(x_1'-x_2',0) \Big|
\leq C\big(|x_1-x_2|+|t_1-t_2|^{1/2}\big)^{1+\alpha}.
$$
By the Lipschitz estimate for $x_n^*$,
$$
|x_1-x_2| \leq |x_1'-x_2'|+ |x_n^*(x_1',t_1)-x_n^*(x_2',t_2)|	\leq C\big (|x_1'-x_2'|+|t_1-t_2|^{1/2}\big).
$$
Dividing by $n_{x_2,t_2}\ge 2/3$, we conclude that
$$
\Big|x_n^*(x_1',t_1)-x_n^*(x_2',t_2) +\frac{\vep_{x_2,t_2}}{n_{x_2,t_2}} \cdot(x_1'-x_2',0) \Big| \leq C\big(|x_1'-x_2'|+|t_1-t_2|^{1/2}\big)^{1+\alpha}.
$$
In particular,
$$
\nabla_{x'} x_n^*(x_2',t_2) =-\frac{\vep_{x_2,t_2}'}{n_{x_2,t_2}}.
$$
By Corollary~\ref{cor:normal_Holder}, this gradient is $C^{0,\alpha}$. Therefore, $x_n^*\in C^{1,\alpha}$, in the parabolic sense, so $\mathcal R$ is locally a $C^{1,\alpha}$ hypersurface. 
\end{proof}

\appendix
\section{Flat transmission problems}
\label{a:flattrans}

In this section, we collect two lemmas regarding viscosity solutions of flat transmission problems that we use in the proof of the improvement of flatness (Theorem~\ref{thm:IF}).

\begin{lem}\label{lem:Pucci_implies_transmission}
Let $w$ be a viscosity solution of
    $$
    \begin{cases}
\partial_t w- F^-(D^2 w)  = 0 & \text{ in } Q_1^+,\\
\partial_t w-F^+(D^2 w)= 0 & \text{ in } Q_1^-,\\
\partial_t w- F^-(D^2 w)  \geq 0 & \text{ in } Q_1,\\
 \partial_t w - F^+(D^2 w)   \leq 0 & \text{ in } Q_1.
 \end{cases}
    $$
    Then, $w$ satisfies the transmission condition 
    $$\partial_{x_n}^+w = \partial_{x_n}^-w \quad \text{ on  } \{x_n = 0\},$$
    in the viscosity sense.
    In other words, for $P$ a polynomial, quadratic in $x'$ and affine in $t$, if
    $$\varphi = P(x',t) + px_n^+ - qx_n^-,$$ 
     touches $w$ at $(x_0,t_0) \in \{x_n = 0\}\cap Q_1$ from above (resp.~below), then 
    $$p \geq q \qquad (resp.~p \leq q).$$
\end{lem}

\begin{proof}
    We proceed by contradiction. Assume without loss of generality that $(x_0,t_0) = (0,0)$, $\varphi(0,0) = w(0,0)$, and $\varphi \geq w$ in $Q_\delta$ for some $\delta > 0$, but $p < q$. Now,
    \begin{align*}
        w(x,t) &\leq \varphi(x,t) = P(x',t) + px_n^+ - qx_n^- = P(x',t) + \frac{p+q}{2}x_n - \frac{q-p}{2}|x_n|\\
        &\leq P(x',t) + \frac{p+q}{2}x_n - \frac{q-p}{2\delta}x_n^2 =: \psi(x,t).
    \end{align*}
    Note that $\psi$ touches $w$ from above at $(0,0)$ as well, and that
    $$\partial_t\psi - F^+(D^2\psi) \geq -C + \frac{\lambda(q-p)}{\delta} > 0,$$
    choosing $\delta$ smaller if needed, contradicting that $\partial_t w - F^+(D^2w) \leq 0$, in the viscosity sense.
\end{proof}

Combining Lemma~\ref{lem:Pucci_implies_transmission} with \cite[Proposition 4.20]{JSC26}, we obtain the following interior regularity estimate.

\begin{lem}\label{lem:blowup_C1_alpha}
    Let $w$ be a bounded viscosity solution of
    $$
    \begin{cases}
\partial_t w- F^-(D^2 w)  = 0 & \text{ in } Q_1^+,\\
\partial_t w-F^+(D^2 w)= 0 & \text{ in } Q_1^-,\\
\partial_t w- F^-(D^2 w)  \geq 0 & \text{ in } Q_1,\\
 \partial_t w - F^+(D^2 w)   \leq 0 & \text{ in } Q_1.
 \end{cases}
    $$
    Then,
    $$\|w\|_{C^{1,\alpha}(Q_{1/2})} \leq C_0\|w\|_{L^\infty(Q_1)},$$
    where $\alpha$ and $C_0$ are positive constants depending only on $n$, $\lambda$ and $\Lambda$.
\end{lem}

\section{Global solutions to the heat equation}\label{a:heateqnsol}

In this section, we state a simple lemma needed in the proof of Theorem~\ref{thm:uniformest} regarding the behavior of global solutions to the heat equation.
For $z>0$, we recall the $\Gamma$-function
$$
\Gamma(z) := \int_0^\infty r^{z-1} e^{-r}\, dr.
$$

\begin{lem}\label{lem:globalsol}
Let $\Phi$ be the function given by
$$
\Phi(x,t)=   \frac{1}{(4\pi t)^{n/2}} \int_{\Rn} e^{-\tfrac{|y-x|^2}{4t}} |y| \, dy, \quad \text{ for } (x,t)\in \Rn \times (0,\infty).
$$
 Then $\Phi$ satisfies $\Phi_t - \Delta \Phi =0$ in $\Rn\times (0,\infty)$, and $\Phi(x,0)=|x|$ for $x\in \Rn$.
Moreover,
    $$
    \Phi(0,t)= 2 \frac{\Gamma\left(\frac{n+1}{2}\right)}{\Gamma\left(\frac n2\right)}t^{1/2}, \qquad \text{for all } t>0.
    $$
\end{lem}

\begin{proof}
By definition, $\Phi$ is the convolution of the heat kernel and $|x|$. Hence, $\Phi$ satisfies $\Phi_t - \Delta \Phi=0$ in $\Rn \times (0,\infty)$. Taking $t\to 0$, we also have $\Phi(x,0)=|x|$ for $x\in \Rn$.

Moreover, for any $t>0$, using polar coordinates, 
$$
\Phi(0,t) = \frac{1}{(4\pi t)^{n/2}} \int_{\Rn} e^{-\frac{|y|^2}{4t}} |y|\,dy
  = \frac{|\mathbb{S}^{n-1}|}{(4\pi t)^{n/2}} \int_0^\infty e^{-\frac{r^2}{4t}} r^n\,dr.
$$
We have that 
$$
\int_0^\infty
e^{-\frac{r^2}{4t}} r^n\,dr
=
2^n t^{\frac{n+1}{2}}
\Gamma\left(\tfrac{n+1}{2}\right) \quad \text{ and } \quad |\mathbb{S}^{n-1}| = \frac{2\pi^{n/2}}{\Gamma\left(\frac n2\right)}.
$$
Therefore,
$$
\Phi(0,t) =2 \frac{\Gamma\left(\frac{n+1}{2}\right)}{\Gamma\left(\frac n2\right)} t^{1/2}.
$$
\end{proof}

\begin{rem} \label{rem:globalsol}
    For any $A,B, t_0\in \R$ and $x_0\in \Rn$, taking 
    $$\tilde \Phi(x,t)=A+B\Phi(x-x_0,t-t_0), \quad \text{ for } (x,t)\in \Rn\times (t_0,\infty),$$ 
    by linearity and translation invariance, it follows that $\tilde\Phi$ is a solution to the heat equation in $\R^n\times (t_0,\infty)$,
    with $\tilde \Phi(x,t_0)= A + B |x-x_0|$, such that
    $$
    \tilde \Phi(x_0,t)= A + 2B \frac{\Gamma\left(\frac{n+1}{2}\right)}{\Gamma\left(\frac n2\right)} (t-t_0)^{1/2}, \qquad \text{for all } t>t_0.
    $$
\end{rem}

\bigskip
\subsection*{Acknowledgements}
We are deeply grateful to Stefania Patrizi and Monica Torres for proposing this project during the ``Research Collaboration Workshop for Women in Analysis,'' held at the Banff International Research Station (BIRS) in Canada from May 11 to 16, 2025. We would like to specifically acknowledge the organizers, Donatella Danielli and Irina Mitrea and the third-named author. We are very thankful for being able to participate in such a wonderful workshop and that such an opportunity to foster research collaboration and networking among women in analysis exists. We also thank BIRS for its generous hospitality.

\subsection*{Funding}

E.C. was partially supported by NSF DMS 1954545, NSF DMS 2602489, the AMS-Simons Travel Grant, and by the Simons Initiative on the Geometry of Flows (Grant Award ID BD-Targeted-00017375-ME).

M.S.C. acknowledges financial support from the grants JDC2023-050724-I and PID2023-150166NB-I00 
(MCIN/AEI/10.13039/501100011033), funded by the Ministry of Science and Innovation (Spain), through the Agencia Estatal de Investigaci\'on (AEI).

C.T.L. has received funding from the European Research Council (ERC) under the Grant Agreement No. 862342, from AEI project PID2024-156429NB-I00 (Spain), and from the Grant CEX2023-001347-S funded by MICIU/AEI/10.13039/501100011033 (Spain).

\bibliographystyle{alpha}
\bibdata{references}
\bibliography{references}

@article{CIL92,
  author = {Crandall, Michael G. and Ishii, Hitoshi and Lions, Pierre-Louis},
  title = {User's guide to viscosity solutions of second order partial differential equations},
  journal = {Bulletin of the American Mathematical Society},
  volume = {27},
  number = {1},
  pages = {1--67},
  year = {1992},
  doi = {10.1090/S0273-0979-1992-00266-5},
}

@article{CKS00,
  author = {Crandall, Michael G. and Kocan, Maciej and {\'S}wiech, Andrzej},
  title = {{$L^p$}-theory for fully nonlinear uniformly parabolic equations: Parabolic equations},
  journal = {Communications in Partial Differential Equations},
  volume = {25},
  number = {11--12},
  pages = {1997--2053},
  year = {2000},
  publisher = {Taylor \& Francis},
}

@article{CPQT19,
  author = {Caffarelli, Luis A. and Patrizi, Stefania and Quitalo, Veronica and Torres, Monica},
  title = {Regularity of interfaces for a {Pucci} type segregation problem},
  journal = {Annales de l'Institut Henri Poincar{\'e} C, Analyse Non Lin{\'e}aire},
  volume = {36},
  number = {4},
  pages = {939--975},
  year = {2019},
  doi = {10.1016/j.anihpc.2018.11.002},
}

@book{CC95,
  author = {Caffarelli, Luis A. and Cabr{\'e}, Xavier},
  title = {Fully nonlinear elliptic equations},
  series = {American Mathematical Society Colloquium Publications},
  volume = {43},
  year = {1995},
  publisher = {American Mathematical Society},
  address = {Providence, RI},
  isbn = {0-8218-0437-5},
}

@article{DS11,
  author = {De Silva, Daniela},
  title = {Free boundary regularity for a problem with right hand side},
  journal = {Interfaces and Free Boundaries},
  volume = {13},
  number = {2},
  pages = {223--238},
  year = {2011},
  doi = {10.4171/IFB/255},
}

@article{IL90,
  author = {Ishii, Hitoshi and Lions, Pierre-Louis},
  title = {Viscosity solutions of fully nonlinear second-order elliptic partial differential equations},
  journal = {Journal of Differential Equations},
  volume = {83},
  number = {1},
  pages = {26--78},
  year = {1990},
  doi = {10.1016/0022-0396(90)90068-Z},
}

@article{WI92,
  author = {Wang, Lihe},
  title = {On the regularity theory of fully nonlinear parabolic equations. {I}},
  journal = {Communications on Pure and Applied Mathematics},
  volume = {45},
  number = {1},
  pages = {27--76},
  year = {1992},
  doi = {10.1002/cpa.3160450103},
}

@article{KSC26,
  author = {Kriventsov, Dennis and Soria-Carro, Mar{\'i}a},
  title = {A parabolic free transmission problem: flat free boundaries are smooth},
  journal = {Calculus of Variations and Partial Differential Equations},
  volume = {65},
  number = {1},
  pages={35 pp.},
  eid = {35},
  year = {2026},
  doi = {10.1007/s00526-025-03188-y},
}

@article{WII92,
  author = {Wang, Lihe},
  title = {On the regularity theory of fully nonlinear parabolic equations. {II}},
  journal = {Communications on Pure and Applied Mathematics},
  volume = {45},
  number = {2},
  pages = {141--178},
  year = {1992},
  doi = {10.1002/cpa.3160450202},
}

@book{CS05,
  author = {Caffarelli, Luis A. and Salsa, Sandro},
  title = {A geometric approach to free boundary problems},
  series = {Graduate Studies in Mathematics},
  volume = {68},
  year = {2005},
  publisher = {American Mathematical Society},
}

@article{AS16,
  author = {Allen, Mark and Shi, Wenhui},
  title = {The two-phase parabolic {S}ignorini problem},
  journal = {Indiana University Mathematics Journal},
  volume = {65},
  number = {2},
  pages = {727--741},
  year = {2016},
  publisher = {JSTOR},
  doi = {10.1512/iumj.2016.65.5824},
}

@article{JSC26,
  author = {Jesus, David and Soria-Carro, Mar{\'i}a},
  title = {Fully nonlinear parabolic fixed transmission problems.},
  journal = {Nonlinear Analysis},
  volume = {264},
  pages = {Paper No. 114004, 26 pp.},
  year = {2026},
}

@article{CL08,
  author = {Caffarelli, Luis A. and Lin, Fang-Hua},
  title = {Singularly perturbed elliptic systems and multi-valued harmonic functions with free boundaries},
  journal = {Journal of the American Mathematical Society},
  volume = {21},
  number = {3},
  pages = {847--862},
  year = {2008},
}

@article{CKL09,
  author = {Caffarelli, Luis A. and Karakhanyan, Aram L. and Lin, Fang-Hua},
  title = {The geometry of solutions to a segregation problem for nondivergence systems},
  journal = {Journal of Fixed Point Theory and Applications},
  volume = {5},
  number = {2},
  pages = {319--351},
  year = {2009},
  doi = {10.1007/s11784-009-0110-0},
}

@article{Q13,
  author = {Quitalo, Veronica},
  title = {A free boundary problem arising from segregation of populations with high competition},
  journal = {Archive for Rational Mechanics and Analysis},
  volume = {210},
  number = {3},
  pages = {857--908},
  year = {2013},
}

@article{CDHMN04,
  author = {Crooks, Elaine C. M. and Dancer, Edward N. and Hilhorst, Danielle and Mimura, Masayasu and Ninomiya, Hirokazu},
  title = {Spatial segregation limit of a competition-diffusion system with {D}irichlet boundary conditions},
  journal = {Nonlinear Analysis: Real World Applications},
  volume = {5},
  pages = {645--665},
  year = {2004},
}

@article{DD03,
  author = {Dancer, Edward N. and Du, Yi Hong},
  title = {On a free boundary problem arising from population biology},
  journal = {Indiana University Mathematics Journal},
  volume = {52},
  number = {1},
  pages = {51--67},
  year = {2003},
}

@article{DD94,
  author = {Dancer, Edward N. and Du, Yi Hong},
  title = {Competing species equations with diffusion, large interactions, and jumping nonlinearities},
  journal = {Journal of Differential Equations},
  volume = {114},
  pages = {434--475},
  year = {1994},
}

@article{DD95II,
  author = {Dancer, Edward N. and Du, Yi Hong},
  title = {Positive solutions for a three-species competition system with diffusion. {II.} {T}he case of equal birth rates},
  journal = {Nonlinear Analysis},
  volume = {24},
  number = {3},
  pages = {359--373},
  year = {1995},
}

@article{DD95I,
  author = {Dancer, Edward N. and Du, Yi Hong},
  title = {Positive solutions for a three-species competition system with diffusion. {I.} {G}eneral existence results},
  journal = {Nonlinear Analysis},
  volume = {24},
  number = {3},
  pages = {337--357},
  year = {1995},
}

@misc{GS25,
  author = {Goering, Max and Skorobogatova, Anna},
  title = {On the geometric properties of multi-operator two-phase elliptic measure},
  year = {2025},
  eprint = {2509.04189},
  archivePrefix = {arXiv},
  note = {arXiv preprint},
}

@article{CTV05III,
  author = {Conti, Monica and Terracini, Susanna and Verzini, Gianmaria},
  title = {Asymptotic estimates for the spatial segregation of competitive systems},
  journal = {Advances in Mathematics},
  volume = {195},
  number = {2},
  pages = {524--560},
  year = {2005},
}

@article{CTV05II,
  author = {Conti, Monica and Terracini, Susanna and Verzini, Gianmaria},
  title = {A variational problem for the spatial segregation of reaction-diffusion systems},
  journal = {Indiana University Mathematics Journal},
  volume = {54},
  number = {3},
  pages = {779--815},
  year = {2005},
}

@article{CTV05I,
  author = {Conti, Monica and Terracini, Susanna and Verzini, Gianmaria},
  title = {Nehari's problem and competing species systems},
  journal = {Annales de l'Institut Henri Poincar{\'e} C, Analyse Non Lin{\'e}aire},
  volume = {19},
  number = {6},
  pages = {871--888},
  year = {2002},
  doi = {10.1016/S0294-1449(02)00104-X},
}

@article{OV26,
  author = {Ognibene, Roberto and Velichkov, Bozhidar},
  title = {A survey on the optimal partition problem},
  journal = {La Matematica},
  volume = {5},
  number = {1},
  eid = {6},
  year = {2026},
  doi = {10.1007/s44007-025-00176-8},
}

@article{DHMP99,
  author = {Dancer, Edward N. and Hilhorst, Danielle and Mimura, Masayasu and Peletier, Lambertus A.},
  title = {Spatial segregation limit of a competition-diffusion system},
  journal = {European Journal of Applied Mathematics},
  volume = {10},
  number = {2},
  pages = {97--115},
  year = {1999},
}

@article{DWZ11,
  author = {Dancer, Edward N. and Wang, Kelei and Zhang, Zhitao},
  title = {Uniform {H}\"{o}lder estimate for singularly perturbed parabolic systems of {B}ose-{E}instein condensates and competing species},
  journal = {Journal of Differential Equations},
  volume = {251},
  number = {10},
  pages = {2737--2769},
  year = {2011},
}

@article{DWZ12,
  author = {Dancer, Edward N. and Wang, Kelei and Zhang, Zhitao},
  title = {The limit equation for the {G}ross--{P}itaevskii equations and {S.} {T}erracini's conjecture},
  journal = {Journal of Functional Analysis},
  volume = {262},
  number = {3},
  pages = {1087--1131},
  year = {2012},
}

@article{AST21,
  author = {Audrito, Alessandro and Serra, Enrico and Tilli, Paolo},
  title = {A minimization procedure to the existence of segregated solutions to parabolic reaction-diffusion systems},
  journal = {Communications in Partial Differential Equations},
  volume = {46},
  number = {12},
  pages = {2268--2287},
  year = {2021},
  doi = {10.1080/03605302.2021.1931884},
}

@article{STVZ25,
  author = {Salort, Ariel and Terracini, Susanna and Verzini, Gianmaria and Zilio, Alessandro},
  title = {Rotating spirals in segregated reaction-diffusion systems},
  journal = {Analysis \& PDE},
  volume = {18},
  number = {3},
  pages = {549--590},
  year = {2025},
}

@article{CPQ17,
  author = {Caffarelli, Luis A. and Patrizi, Stefania and Quitalo, Veronica},
  title = {On a long range segregation model},
  journal = {Journal of the European Mathematical Society},
  volume = {19},
  number = {12},
  pages = {3575--3628},
  year = {2017},
}

@article{CPT26,
  author = {Chuah, Howen and Patrizi, Stefania and Torres, Monica},
  title = {A nonlinear model for long-range segregation},
  journal = {La Matematica},
  volume = {5},
  number = {1},
  eid = {20},
  year = {2026},
  doi = {10.1007/s44007-026-00198-w},
}

@article{H07,
  author = {Han, Qing},
  title = {Nodal Sets of Harmonic Functions},
  journal = {Pure and Applied Mathematics Quarterly},
  volume = {3},
  number = {3},
  pages = {647--688},
  year = {2007},
  doi = {10.4310/PAMQ.2007.v3.n3.a2},
}

@article{C87,
  author = {Caffarelli, Luis A.},
  title = {A {H}arnack inequality approach to the regularity of free boundaries. {Part I}: Lipschitz free boundaries are {$C^{1,\alpha}$}},
  journal = {Revista Matem{\'a}tica Iberoamericana},
  volume = {3},
  number = {2},
  pages = {139--162},
  year = {1987},
  doi = {10.4171/RMI/47},
}

@article{C89,
  author = {Caffarelli, Luis A.},
  title = {A {H}arnack inequality approach to the regularity of free boundaries. {Part II}: {F}lat free boundaries are {L}ipschitz},
  journal = {Communications on Pure and Applied Mathematics},
  volume = {42},
  number = {1},
  pages = {55--78},
  year = {1989},
  doi = {10.1002/cpa.3160420105},
}

@article{PS23,
  author = {Pimentel, Edgard A. and Santos, Makson S.},
  title = {Fully nonlinear free transmission problems},
  journal = {Interfaces and Free Boundaries},
  volume = {25},
  number = {3},
  pages = {325--342},
  year = {2023},
  doi = {10.4171/IFB/489},
}

@article{PSw22,
  author = {Pimentel, Edgard A. and {\'S}wi{\v{e}}ch, Andrzej},
  title = {Existence of solutions to a fully nonlinear free transmission problem},
  journal = {Journal of Differential Equations},
  volume = {320},
  pages = {49--63},
  year = {2022},
  doi = {10.1016/j.jde.2022.02.057},
}

@article{ST23,
  author = {Soave, Nicola and Terracini, Susanna},
  title = {An anisotropic monotonicity formula, with applications to some segregation problems},
  journal = {Journal of the European Mathematical Society},
  volume = {25},
  number = {9},
  pages = {3727--3765},
  year = {2023},
}

@article{AT15,
  author = {Amaral, Marcelo and Teixeira, Eduardo V.},
  title = {Free transmission problems},
  journal = {Communications in Mathematical Physics},
  volume = {337},
  number = {3},
  pages = {1465--1489},
  year = {2015},
}

@article{AM13,
  author = {Andersson, John and Mikayelyan, Hayk},
  title = {The zero level set for a certain weak solution, with applications to the {Bellman} equations},
  journal = {Transactions of the American Mathematical Society},
  volume = {365},
  number = {5},
  pages = {2297--2316},
  year = {2013},
}

@article{CDSS18,
  author = {Caffarelli, Luis A. and De Silva, Daniela and Savin, Ovidiu},
  title = {Two-phase anisotropic free boundary problems and applications to the {Bellman} equation in {2D}},
  journal = {Archive for Rational Mechanics and Analysis},
  volume = {228},
  number = {2},
  pages = {477--493},
  year = {2018},
}

@article{KLS17,
  author = {Kim, Sunghan and Lee, Ki-Ahm and Shahgholian, Henrik},
  title = {An elliptic free boundary arising from the jump of conductivity},
  journal = {Nonlinear Analysis},
  volume = {161},
  pages = {1--29},
  year = {2017},
}

@article{BDS25,
  author = {Beck, Thomas and De Silva, Daniela and Savin, Ovidiu},
  title = {Two-phase problems: {P}erron solutions and regularity of the {N}eumann problem in convex cones},
  journal = {International Mathematics Research Notices},
  volume = {2025},
  number = {11},
  eid = {rnaf152},
  year = {2025},
  doi = {10.1093/imrn/rnaf152},
}

@article{DSS19,
  author = {De Silva, Daniela and Savin, Ovidiu},
  title = {Global solutions to nonlinear two-phase free boundary problems},
  journal = {Communications on Pure and Applied Mathematics},
  volume = {72},
  number = {10},
  pages = {2031--2062},
  year = {2019},
  doi = {10.1002/cpa.21811},
}

\end{document}